\documentclass[12pt]{amsart}
\usepackage{latexsym}
\usepackage{amssymb}
\usepackage[cp850]{inputenc}
\usepackage{epsfig}
\usepackage{amsthm}
\usepackage{amscd}
\usepackage{amsmath}
\usepackage{amsfonts}
\usepackage{graphics}
\usepackage[all]{xy}
\newtheorem{theorem}{Theorem}[section] 
\newtheorem{lemma}[theorem]{Lemma}
\newtheorem{corollary}[theorem]{Corollary} 
\newtheorem{prop}[theorem]{Proposition}
\newtheorem{claim}[theorem]{Claim}
\newtheorem{fact}[theorem]{Fact}
\newtheorem*{theorem*}{Theorem} 
\newtheorem*{corollary*}{Corollary}
\theoremstyle{definition}
\newtheorem{example}[theorem]{Example}

\newtheorem{definition}[theorem]{Definition}
\newtheorem*{remark*}{Remark}

\newtheorem*{definition*}{Definition}
\newtheorem*{example*}{Example}
\newtheorem*{namedtheorem}{\theoremname}

\newcommand{\theoremname}{testing}
\newenvironment{named}[1]{\renewcommand{\theoremname}{#1}\begin{namedtheorem}}{\end{namedtheorem}}
\theoremstyle{remark}

\newcommand{\BC}{\mathbb C} \newcommand{\BH}{\mathbb H}
\newcommand{\BR}{\mathbb R} \newcommand{\BD}{\mathbb D}
\newcommand{\BN}{\mathbb N} 
\newcommand{\BS}{\mathbb S} \newcommand{\BZ}{\mathbb Z}

\newcommand{\CA}{\mathcal A} 
\newcommand{\CE}{\mathcal E} 

\newcommand{\CK}{\mathcal K} \newcommand{\CL}{\mathcal L}
\newcommand{\CM}{\mathcal M} \newcommand{\CN}{\mathcal N}

 \newcommand{\CT}{\mathcal T}

\newcommand{\actson}{\curvearrowright}

\DeclareMathOperator{\PSL}{PSL} 
\DeclareMathOperator{\Id}{Id} 
\DeclareMathOperator{\Hom}{Hom} 

\DeclareMathOperator{\tr}{Tr}

\DeclareMathOperator{\length}{length}

\DeclareMathOperator{\dist}{d}

\DeclareMathOperator{\CC}{CC}

\DeclareMathOperator{\C}{C}
\DeclareMathOperator{\dome}{Dome}
\DeclareMathOperator{\kernel}{ker}
\DeclareMathOperator{\soup}{sup}

\DeclareMathOperator{\Mod}{Mod}

\newcommand{\PML}{{\mathcal P \mathcal M \mathcal L}}
\newcommand{\Hyp}{{\mathbb H}}

\newcommand{\into}{\hookrightarrow}

\newcommand{\comment}[1]{}

\title{Extending pseudo-Anosov maps into compression bodies}

\author{Ian Biringer}
\address{\hskip-\parindent
        Department of Mathematics \\
        Yale University \\
        PO Box 208283 \\
        New Haven, CT 06520 \\
        USA}
\email{ian.biringer@yale.edu}

\author{Jesse Johnson}
\address{\hskip-\parindent
        Department of Mathematics \\
        Oklahoma State University \\
        Stillwater, OK 74078 \\
        USA}
\email{jjohnson@math.okstate.edu}

\author{Yair Minsky}
\address{\hskip-\parindent
        Department of Mathematics \\
        Yale University \\
        PO Box 208283 \\
        New Haven, CT 06520 \\
        USA}
\email{yair.minsky@yale.edu}

\begin{document}

\begin{abstract}
We show that a pseudo-Anosov map on a boundary component of an irreducible $3$-manifold has a power that partially extends to the interior if and only if its (un)stable lamination is a projective limit of meridians.  The proof is through $3 $-dimensional hyperbolic geometry, and involves an investigation of algebraic limits of convex cocompact compression bodies.
\end{abstract}

\maketitle
\section{ Introduction }
Let $M $ be a compact, orientable and irreducible $3$-manifold with a boundary component $\Sigma $ that is \emph{compressible}, i.e. the inclusion $\Sigma \into M $ is not $\pi_1 $-injective.  Recall that a \it meridian \rm is an essential, simple closed curve on $\Sigma $ that bounds an embedded disk in $M$. The closure $\Lambda (M) \subset \PML (\Sigma) $ of the set of projective measured laminations supported on meridians is called the \it limit set \rm of $M $.   The terminology comes from the fact that $\Lambda (M) $ is the smallest nonempty, closed subset of $\PML (\Sigma) $ that is invariant under the action of the group of homeomorphisms $\Sigma \to \Sigma$ that extend to $M $, see \cite {Masurmeasured}.

Our main result is the following:

\begin{theorem}
\label {Main}
Let $f : \Sigma \to \Sigma$ be a pseudo-Anosov homeomorphism of some compressible boundary component $\Sigma $ of a compact, orientable and irreducible $3 $-manifold $M $.  Then the (un)stable lamination of $f $ lies in $\Lambda (M)$ if and only if $f $ has a power that \emph{partially extends} to $M$.
\end{theorem}
We say $f$ \textit{partially extends} to $M$ if there is a nontrivial compression body $C \subset M$ with exterior boundary $\partial_+ C = \Sigma$ and a homeomorphism $\phi : C \rightarrow C$ such that $\phi|_{\Sigma} = f$.  A \it compression body \rm is a compact, irreducible $3 $-manifold constructed by attaching $2 $-handles to $\Sigma\times [0, 1] $ along a collection of disjoint annuli in $\Sigma \times \{ 0\} $ and $3 $-balls to any boundary components of the result that are homeomorphic to $\BS^ 2 $. We call a compression body \it nontrivial \rm if it is not homeomorphic to a trivial interval bundle.  The \it exterior boundary \rm of a nontrivial compression body is the unique boundary component that $\pi_1 $-surjects.  Note that in the construction above, the exterior boundary is $ \Sigma\times \{ 1\} $.

\subsection*{Remarks}  In the literature, maps $f \in \Mod (\Sigma) $ that do not partially extend to $M $ or have both associated laminations outside $\Lambda (M) $ are often called `generic'  \cite {Namaziheegaard}, \cite {Lackenbyattaching}.  These two conditions have slightly different uses, and in fact \cite {Lackenbyattaching} define genericity as non-extensibility while \cite {Namaziheegaard} uses the condition on laminations. Theorem \ref {Main} reconciles these definitions, and moreover indicates that it is enough to assume that, say, the stable lamination of $f $ lies outside $\Lambda (M) $.

There is no obvious way to sharpen Theorem \ref {Main}, even when $M $ is a handlebody. In Section \ref {examples}, we show that there are pseudo-Anosov maps on the boundary of a handlebody $M $ that extend partially but do not extend to homeomorphisms $M \to M$.  We also show a pseudo-Anosov map $f : \partial M \to \partial M $ can have a power that extends to $M $ without extending even partially itself. 

Observe that as partial extension is symmetric for $f $ and $f^ {- 1 } $, Theorem \ref {Main} shows that $\Lambda (M) $ contains the stable lamination of $f $ if and only if it contains the unstable lamination.  It also suffices to prove Theorem \ref {Main} for, say, the stable lamination of $f $, for otherwise one could replace $f $ with its inverse.

Bonahon \cite {Bonahoncobordism} defined a canonical \it characteristic compression body \rm in $M $ that has exterior boundary $\Sigma $.  It is nontrivial, unique up to isotopy and contains an isotope of any compression body in $M $ with the same exterior boundary.   It has the same limit set as $M $ and the same partial extensions properties for maps $ \Sigma \to \Sigma $.  So, it suffices to prove Theorem \ref {Main} when $M$ is a compression body. Note also that the uniqueness of the characteristic compression body implies that any homeomorphism of $\Sigma $ that extends to $M $ does extend partially.


\vspace{5mm}Before beginning the paper in earnest, we sketch the proof of Theorem \ref {Main}.  The inclusion of forward-looking references lets this double as an outline.

To start with, the `if' direction of Theorem \ref{Main} is trivial.  If $f^ i $ extends to a nontrivial compression body $C \subset M $, then any meridian $\gamma $ for $C $ gives sequences $(f^{ k i} (\gamma)) $ and $(f^{ -k i} (\gamma)) $ of meridians that converge to the stable and unstable laminations of $f $, respectively.   The other direction is much harder, and our argument is based in $3 $-dimensional hyperbolic geometry.  

As remarked above, we may assume that $M $ is a compression body and $\Sigma $ is its exterior boundary.  The first part of the argument is a construction:

\begin{center}
\begin{tabular}{c | c}
Input & Output \\ \hline \hline \vspace {2mm}
$\ \ \  f \in \Mod (\Sigma) \ \ \ $ & \ \  a compression body $C \subset M $ to which $f $ extends. \\
\end{tabular}
\end{center} 
This will work for any $f \in \Mod (\Sigma) $, with the caveat that $C $ may be trivial.

The set of convex cocompact hyperbolic metrics on $M $ is parameterized by the Teichm\"uller space $\CT (\partial M) = \CT (\Sigma) \times \CT (\partial M \setminus \Sigma) $; we create a sequence of such metrics by iterating $f $ on $\CT (\Sigma) $. Using a remarking trick,
we view this as a sequence of abstract compression bodies whose exterior boundaries are marked by $\Sigma $ in such a way that the associated sequence in $\CT (\Sigma) $ is constant.  The markings determine a sequence of holonomy representations $\rho_i : \pi_1 \Sigma \to \PSL (2,\BC) $, and we let $\CA_f \subset \Hom (\pi_1 \Sigma,\PSL (2,\BC) )$ be the set of its accumulation points. For each $\rho \in \CA_f $, the quotient $\BH^ 3 / \rho (\pi_1\Sigma) $ is homeomorphic to the interior of a compression body with exterior boundary
$\Sigma$ (Section 3). The kernels $\{ \ker \rho \ | \ \rho \in \CA_f \} $ are ordered by inclusion,
and we show (Section 5) that the set of minimal elements is finite and $f_* $-invariant.  Some power $f_*^ i $ then fixes each minimal $\ker \rho $, so $f ^ i$ extends to the quotients $\BH^ 3 / \rho (\pi_1\Sigma) $.   Finally, we show that one of these quotients embeds as a subcompression body $C\subset M $.

In the second part of the argument (Section 6 and 7), we show that if $f $ is pseudo-Anosov with stable lamination $\lambda_+ (f) \in\Lambda (M) $ then $C $ is nontrivial.  If it were trivial, we would have a sequence of hyperbolic compression bodies with boundaries marked by $\Sigma $ converging to a hyperbolic $ \Sigma \times \BR $.  The disk sets of these compression bodies must then go to infinity in the curve complex $\mathcal C (\Sigma) $, by Section 7. These disk sets are constructed (Section 3) to be $f^ { - 1 } $-iterates of the disk set $\mathcal D (M) $, so a final remarking implies that every \it forward \rm orbit of $f \actson \mathcal C (\Sigma) $ strays arbitrarily far from $\mathcal D (M) $.   Masur-Minsky's quasi-convexity of disk sets (see Proposition \ref{boundeddistance}) then shows that no forward orbit of $f \actson \mathcal C (\Sigma) $ can limit into the Gromov boundary of $\mathcal D (M) $.  But these orbits all limit to the support of $\lambda_+ (f)$ in $\partial_\infty \mathcal  C (\Sigma) $, and $\partial_\infty \mathcal D (M) $ consists of the supports of elements of $\Lambda (M) $.  So, it follows that $\lambda_+ (f) \notin \Lambda (M) $.



\vspace {2mm}

The authors would like to thank Dick Canary, Cyril Lecuire, Justin Malestein and Juan Souto for helpful conversations.  The first author was partially supported by NSF postdoctoral fellowship DMS-0902991, and the second was partially supported by NSF postdoctoral fellowship DMS-0602368.

\section { Preliminaries }

\label{preliminaries}

This section reviews some necessary background for our work.   It will begin with some definitions from coarse geometry. After that, we will discuss measured laminations, the curve and disk complexes, and some qualities of the action of the mapping class group $\Mod (\Sigma) $ on the curve complex.  We then transition into hyperbolic $3$-manifolds, discussing the classification of ends, the relationship between the conformal and convex core boundaries, and algebraic convergence.  Some good references for this material are \cite {Bridsonmetric}, \cite{Bleilerautomorphisms}, \cite{Thurstongeometry} and \cite {Matsuzakihyperbolic}.

\subsection{Hyperbolicity and the boundary at infinity}
Given a metric space $X$ and a base point $x \in X $, recall that the \it Gromov product \rm of two points $y, z \in X $ is defined by
$$\left < y | z \right >_x = \frac 12 \big (\dist (y, x) + \dist (z, x) - \dist (y,z) \big). $$ Then $X $ is \it $\delta $-hyperbolic \rm if for all $y,z,w \in X $ we have $$\left <y | z \right >_x \geq\min\{\, \left <y | w\right >_x, \left <z | w \right >_x \} - \delta .  $$  If $X $ is a geodesic space, this definition of $\delta $-hyperbolicity is equivalent to the condition that all geodesic triangles are $\delta $-thin \cite {Bridsonmetric}.  

A definition of Gromov assigns to each $\delta $-hyperbolic space $X $ a natural boundary $\partial_\infty X $.  Namely, a sequence $(y_i)$ in $X $ is called \it admissible \rm if $$\lim_{i,j \to \infty } \left <y_i | y_j\right > = \infty , $$  and the Gromov boundary $\partial_\infty X $ is obtained from the set of admissible sequences in $X $ by identifying two sequences if their interleave is still admissible.  One can extend the Gromov product $\left < \cdot | \cdot \right >_x $ to $X \cup \partial_\infty X $:  if $\bar y =(y_i) $ and $\bar z =(z_i) $ are two admissible sequences, then we set $\left <\bar y | \bar z \right > = \lim_{ i \to \infty} \left < y_i | z_i \right >$.  A topology on $X \cup \partial_\infty X $ extending that of $X $ can then be defined by letting, for a sequence $(y_i) \subset X \cup \partial_\infty X $ and a point $\bar z \in \partial_\infty X $,
 $$\lim_{ i \to \infty } y_i=\bar z \Longleftrightarrow \lim_{ i \to \infty } \left < y_i | \bar z \right > = \infty . $$

\subsection{ Laminations } Throughout the following, let $\Sigma $ be a closed orientable surface of genus at least $2 $ and fix a hyperbolic structure on $\Sigma $. A \it geodesic lamination \rm on $\Sigma $ is a closed subset $\lambda \subset \Sigma $ that is the union of disjoint, simple geodesics.  Geodesic laminations often carry a \it transverse measure\rm: that is, a function $$m : \{a : [0, 1] \to \Sigma\ |\ a \text { is transverse to } \lambda \, \} \longrightarrow \BR_{\geq 0 } $$ that is additive under concatenation of arcs, vanishes on arcs that do not intersect $\lambda $, and assigns two arcs the same value if they differ by an ambient isotopy of $\Sigma$ that leaves $\lambda $ invariant.  The \it support \rm of a transverse measure is the smallest geodesic lamination that carries it; a geodesic lamination equipped with a transverse measure of full support is called a \it measured lamination\rm.  

The set of all measured laminations on $\Sigma $ is written $\CM \CL (\Sigma) $ and is usually considered with the weak$^*$-topology on transverse measures.  The space $\CM \CL (\Sigma) $ admits a natural $\BR_+ $-action through scaling transverse measures.  The quotient by this action is the projective measured laminations space $\PML (\Sigma) $.  Thurston has shown \cite {Thurstongeometry} that $\PML (\Sigma) $ is homeomorphic to a sphere of dimension $6g-7 $, where $g $ is the genus of $\Sigma $.

Another result of Thurston \cite {Thurstongeometry} is that measured laminations supported on unions of closed geodesics are dense in $\CM \CL (\Sigma) $.  In fact, the (weighted) geometric intersection number of two such laminations extends continuously (again, \cite {Thurstongeometry}) to a function $$i : \CM \CL (\Sigma) \times \CM \CL (\Sigma) \to \BR_{ \geq 0 }, $$ which gives the \it intersection number \rm of two measured laminations.

A measured lamination $\lambda $ is called \it filling \rm if $i (\lambda,\mu) >0 $ for any measured lamination $\mu $ with different support. The support of a filling measured lamination on $\Sigma $ is called an \it ending lamination \rm on $\Sigma $.  The set of all ending laminations is written $\CE \CL (\Sigma) $; it is considered with the quotient topology coming from the usual weak$^*$-topology on filling measured laminations.

\subsection{ The complex of curves } As before, let $\Sigma $ be a closed orientable surface of genus at least $2 $.  The \it complex of curves \rm on $\Sigma $, written $\mathcal C (\Sigma) $, is the simplicial complex defined as follows. The vertices of $\mathcal C (\Sigma) $ correspond to homotopy classes of essential simple closed curves on $\Sigma $, and a set of vertices forms a simplex when there is a set of pairwise disjoint representative curves on $\Sigma $.  One can metrize $\mathcal C (\Sigma) $ with the path metric whose restriction to each simplex is isometric to a regular Euclidean simplex with side lengths $1 $.  

Masur and Minsky \cite{Masurgeometry1} have proven that the curve complex $\mathcal C (\Sigma) $ is $\delta$-hyperbolic.  By work of Klarreich \cite {Klarreichboundary}, its Gromov boundary $\partial_\infty \mathcal C (\Sigma) $  is homeomorphic to the space of ending laminations $\CE \CL (\Sigma) $.  To understand the topology on $\mathcal C (\Sigma) \cup \partial_\infty \mathcal C (\Sigma) $, note that a point in $\mathcal C (\Sigma) $ can be considered as a measured lamination consisting of a simple closed curve with weight $1$.

\begin{theorem}[Klarreich \cite {Klarreichboundary}] A sequence $(\gamma_i) $ in $\mathcal C (\Sigma) $ converges to an ending lamination $\lambda \in \partial_\infty \mathcal C (\Sigma) $ if and only if there are weights $c_i \in \BR $ and a transverse measure $m $ on $\lambda $ such that $c_i \gamma_i \to (\lambda,m) $ in $\CM \CL (\Sigma) $.\label {Claridge}
\end{theorem}

\subsection{The disc complex}  Assume that $\Sigma$ is a boundary component of some compact irreducible $3$-manifold $M $.  A \it meridian \rm on $\Sigma $ is an essential simple closed curve on $\Sigma $ that bounds an embedded disc in $M$.  The subcomplex of $\mathcal C (\Sigma) $ spanned by all meridian curves is called the \it disc complex \rm $\mathcal D  (M) $.

The following theorem of Masur and Minsky \cite {Masurquasi-convexity} is central to our work.
\begin{theorem}\label {quasi-convexity}
The disk set $\mathcal D (M) $ is a quasi-convex subset of $\mathcal C (\Sigma) $.
\end{theorem}

 By Theorem \ref{Claridge}, its Gromov boundary $\partial_\infty \mathcal D (M) $ is the subset of $\partial_\infty \mathcal C (\Sigma) $ consisting of ending laminations that are the supports of measured laminations that are limits of weighted meridians in $\CM \CL (\Sigma) $.  In other words, $\partial_\infty \mathcal D (M) $ is the set of elements in $\partial_\infty \mathcal C (\Sigma) $ that are supports of  measured laminations in the limit set $\Lambda (M)$.

\subsection{The mapping class group and $\mathcal C (\Sigma) $} There is an isometric action $\Mod (\Sigma) \actson \mathcal C (\Sigma) $ obtained by extending the natural action on the vertices of $\mathcal C (\Sigma)$ to the higher dimensional cells.  Periodic and reducible elements of $\Mod (\Sigma) $ can easily be shown to act \it elliptically\rm, in the sense that they have a bounded orbit.  Orbits of pseudo-Anosov maps $[f] \in \Mod (\Sigma) $ are always unbounded; moreover, any forward orbit $(f^i (\gamma) \,  | \, i \in \BN ) $ converges to the attracting lamination $\lambda_+ $ of $f $, regarded as a point in $\partial_\infty \mathcal C(\Sigma)  =\CE \CL (\Sigma) $.  

In fact, Masur-Minsky \cite {Masurgeometry1} have shown the following:

\begin{lemma} \label {quasi-geodesics}
Any pseudo-Anosov mapping class $[f] \in \Mod (\Sigma) $ acts {\em hyperbolically} on $\mathcal C (\Sigma) $, meaning that every orbit $(f^ i (\gamma) \, | \, i \in \BZ) $ is a quasi-geodesic.

\end{lemma}

The hyperbolicity of the action of a pseudo-Anosov map on $\mathcal C (\Sigma) $ combines with the quasi-convexity of disc sets (discussed in the previous section) to give the following proposition.  The statement should be no surprise to those familiar with $\delta $-hyperbolic spaces, but we include a full proof to reassure the reader that the local infinitude of $\mathcal C (\Sigma) $ is not problematic.

\begin{prop} Let $\Sigma $ be a boundary component of a compact irreducible $3$-manifold $M $, and consider a pseudo-Anosov map $f : \Sigma \to \Sigma $ with attracting lamination $\lambda_+ \in\Lambda (M) $.  Then for every $\gamma \in \mathcal C (\Sigma) $, $$\sup_{ k = 1, 2,\ldots} \dist (f^{ k} (\gamma),\mathcal D (M)) < \infty. $$ \label {boundeddistance}
\end{prop}
\begin{proof}
Since $\lambda_+ \in  \Lambda (M)$, its support $[\lambda_+] $ lies in the boundary $\partial_\infty \mathcal D (M)$ of the disc complex.  So, we can choose a sequence of meridians $(m_i) \subset \mathcal D (M)$ that converges to $ [\lambda_+]$.  Since $f ^ {i } (\gamma) $ and $m_i $ converge to the same point at infinity, the Gromov product $$\left <f ^ {i } (\gamma)\, | \,m_i\right >_\gamma = \frac { 1 }  { 2 } \big(\dist (f ^ {  i } (\gamma),\gamma) +\dist (m_i,\gamma) - \dist (f ^ {  i } (\gamma), m_i)\big) $$ goes to infinity with $i $.  

Fix now some $k \in \BN $.  We claim that the distance from $f^k (\gamma) $ to $\mathcal D (M) $ is bounded above by some constant independent of $k $.  To see this, let $i > k $ and consider the geodesic triangle in $\mathcal C (\Sigma) $ with vertices $\gamma$, $f ^ {  i } (\gamma) $ and $m_i $. Since any $f$-orbit in $\mathcal C (\Sigma) $ is a quasi-geodesic (Lemma \ref {quasi-geodesics} above), the distance from $f ^ { k } (\gamma) $ to the side $[\gamma, f^ i (\gamma)] $ of this triangle is bounded above by some constant independent of $i $ and $k $.  Furthermore, this side lies in the $\delta $-neighborhood of the other two sides.  So in particular, the distance from $f ^ {  k } (\gamma) $ to the other two sides of our triangle is bounded above independent of $i $ and $k $. 

So, either $f^ {k } (\gamma) $ lies close to $[f^ i(\gamma),m_i] $ or close to $[\gamma,m_i] $.  In the latter case, Theorem \ref {quasi-convexity} ensures that the geodesic segment $[\gamma,m_i] $ stays within a bounded distance of $\mathcal D (M) $.  The former case, however, is impossible for large $i $ because the Gromov product $\left <f ^ {i } (\gamma)\, | \,m_i\right >_\gamma $ is approximated up to a uniform additive error by the distance from $\gamma $ to the geodesic segment $[f^ i(\gamma),m_i] $. So since the Gromov product goes to infinity, if $i $ is very large then $[f^ i(\gamma),m_i] $ is very far from $f ^ k (\gamma) $.  
\end{proof}

\subsection {Teichmuller space and $\Mod (\Sigma) $}
\label{Teichmullersection}
Let $\Sigma $ be a closed orientable surface of genus at least $2 $.  The Teichm\"uller space of $\Sigma $ is the quotient space
$$\CT  (\Sigma) = \big \{ \psi \ | \ \psi \text { is a conformal structure on } \Sigma  \big \} / \sim ,$$
where $\psi_1 \sim  \psi_2$ if there is a conformal homeomorphism $ h : (\Sigma, \psi_1) \to ( \Sigma , \psi_2)$ homotopic to the identity map. Here, a \it conformal structure \rm on $\Sigma $ is just a complex structure on $\Sigma $ and a homeomorphism is \it conformal \rm if it is bianalytic, but we use the conformal terminology because it is standard in the subject.  

There is a natural action of $\Mod (\Sigma) $ on $\CT (\Sigma )$, given by pushing forward conformal structures: $$\text { if } [f] \in  \Mod (\Sigma) \text { and } [\psi] \in \CT (\Sigma), \text { then } [f] \big ( \,[\psi] \, \big)=  [f_* \psi]. $$  Here, $f_*\psi $ is the conformal structure on $\Sigma $ whose charts are obtained from the charts of $\psi $ by precomposing with $f^ {-1} $.
We write this in detail because it will be important later not to confuse the action of  an element of $\Mod (\Sigma) $ on $\CT (\Sigma) $ with the action of its inverse.

\subsection{Ends and Ahlfors-Bers theory }
Our proof of Theorem \ref {Main} requires some knowledge of hyperbolic geometry, in particular the classification of ends of hyperbolic $3$-manifolds and the Ahlfors-Bers parameterization of convex cocompact hyperbolic metrics. We recall in this section the relevant parts of the theory.  A more detailed account can be found in \cite {Matsuzakihyperbolic}.

Let $N $ be a complete hyperbolic $3$-manifold with finitely generated fundamental group and no cusps.  The Tameness Theorem of Agol \cite {Agoltameness} and Calegari-Gabai \cite {Calegarishrinkwrapping} states that every end of $N $ has a neighborhood which is a topological product $\Sigma \times (0,\infty) $.  The ends of $N $ admit a geometric classification, depending on their interaction with the \it convex core \rm of $N $.  The convex core is the smallest convex submanifold $\CC (N) $ of $N $ whose inclusion into $N $ is a homotopy equivalence, and an end of $N $ is called \it convex cocompact \rm if it has a neighborhood disjoint from $\CC(N) $ and \emph{degenerate} otherwise.  

Each end $\CE $ of $N $ has an associated \it ending invariant\rm.  Assuming that $\CE $ has a neighborhood homeomorphic to $\Sigma \times (0,\infty)$, its ending invariant will either be a point in the Teichm\"uller space $\CT (\Sigma) $ or a geodesic lamination on $\Sigma $, depending on whether $\CE $ is convex cocompact or degenerate.  We refer the reader to \cite {Matsuzakihyperbolic} for a discussion of the \emph{ending lamination} associated to a degenerate end, and concentrate here on the convex cocompact case.

Assume that $N $ is the quotient of $\BH^ 3 $ by some finitely generated group $\Gamma $.  The {\em limit set} $\Lambda (\Gamma) $ is the smallest nonempty, closed subset of $\BS^ 2_{\infty } $ that is invariant under the boundary action of $\Gamma $.  Its complement is the {\em domain of discontinuity } $\Omega (\Gamma) = \BS^2_{\infty } \setminus \Lambda (\Gamma) $, which is the largest open subset of $\BS^2_{\infty } $ on which $\Gamma $ acts properly discontinuously.  In fact, $\Gamma $ acts properly discontinuously on $\Hyp^3 \cup \Omega (\Gamma) $, and the quotient is a manifold with boundary that has interior $N $ and boundary $\partial_c N =\Omega (\Gamma) / \Gamma $.  The action $\Gamma \actson \Omega (\Gamma)$ is by M\"obius transformations, so its quotient $\partial_c N $ inherits a natural conformal structure and is therefore called the \it conformal boundary \rm of $N $.  The conformal boundary compactifies precisely the convex cocompact ends of $N $;  the component of $\partial_c  N $ that faces a given convex cocompact end is its ending invariant.

One calls the manifold $N $ \emph{convex cocompact} if all of its ends are convex cocompact.  In fact, a convex cocompact hyperbolic $3$-manifold is determined up to isometry by its topology and conformal boundary.  This result is usually known as the Ahlfors-Bers  parameterization.

\begin{theorem}[Thurston, Ahlfors-Bers, see \cite{Matsuzakihyperbolic}] 
\label{AhlforsBers}
Let $N $ be a hyperbolizable $3$-manifold that is the interior of a compact $3$-manifold $\bar N$ with no torus boundary components.  Then there is a bijection $$\{ \text {convex cocompact  hyperbolic metrics on $N $} \} / \text {isotopy} \longrightarrow \CT (\partial\bar{N }), $$ induced from the map taking a convex-cocompact uniformization of $N$ to its conformal boundary.
\end{theorem} 

The term \it hyperbolizable \rm means that $N $ admits some complete hyperbolic metric.  Thurston \cite {Kapovichhyperbolic} showed that a hyperbolizable $N $ admits a convex cocompact metric, while Ahlfors and Bers studied the space of all such metrics up to isotopy.  Here, two metrics on $N $ are \it isotopic \rm if there is a diffeomorphism of $N $ isotopic to the identity map that is an isometry between them.   The well-schooled reader may be uncomfortable with the fact that our space of convex cocompact hyperbolic metrics  is parameterized by $\CT (\partial \bar N)$, rather than some quotient of it.  The reason for this is that usually one considers the space of metrics up to homotopy, rather than isotopy.
\label{ends}

\subsection {Conformal boundaries and the convex core} We describe here the bilipschitz relationship between the conformal boundary of a hyperbolic $3$-manifold and the radius-$r $ boundary of its convex core.   Essentially all the ideas below come from work of Canary and Bridgeman \cite {Bridgemanthurston}, who extended  fundamental work of Epstein and Marden \cite {Epsteinconvex} to the case of $3$-manifolds with compressible boundary.

\label {conformalboundary}


We begin more generally with a hyperbolic domain $\Omega \subset \hat \BC $.  The \it Poincar\'e metric \rm is the metric on $\Omega $ defined infinitesimally by
$$| | v | |_\rho =\inf_{ v' \in T\BH^ 2 } \big \{ | | v' | |_{\BH^ 2 } \ \big | \  D f (v') = v \text { for some conformal } f : \BH^ 2 \to \Omega \big \}. $$ 
It is the unique hyperbolic metric that is conformal on $\Omega $.  We also consider the \it Thurston metric, \rm which is  conformal on $\Omega $ but not hyperbolic.  It has a similar infinitesimal expression:
$$| | v | |_T =\inf_{ v' \in T\BH^ 2 } \big \{ | | v' | |_{\BH^ 2 } \ \big | \  D f (v') = v \text { for some M\"obius } f : \BH^ 2 \to \Omega\big \} .$$  One can define a map to be M\"obius here if it takes circles to circles; alternatively, one can replace $\BH^ 2 $ with its upper half plane model and use restrictions of M\"obius maps of $\hat \BC $.  


The Poincar\'e and Thurston metrics have each been related to a third metric, the \it quasi-hyperbolic \rm metric, by  Beardon-Pommerenke \cite {Beardonpoincare} and Kulkarni-Pinkall \cite [Theorem 7.2] {Kulkarnicanonical}, respectively.  Combining their results gives

\begin{theorem}\label{metriccomparison} If $\Omega $ is a hyperbolic domain in $\hat \BC $ that has injectivity radius at least $\mu > 0 $ in the Poincar\'e metric, then
$$\frac 1 { 2\sqrt 2 (k + \frac {\pi^ 2 } { 2\mu }) } | | v | |_{T } \leq | | v | |_\rho \leq | | v | |_T , $$ where $k = 4 +\log (3 + 2\sqrt 2) \approx 5.76. $
\end{theorem}

Following Canary-Bridgeman \cite {Bridgemanthurston}, let $\dome (\Omega) \subset \BH^ 3 $ be the boundary of the hyperbolic convex hull of the complement of $\Omega $ in $\hat \BC = \partial_\infty \BH^ 3 $.  Fixing $r > 0 $, we also let $\dome_r (\Omega) $ be the boundary of the radius-$r $ neighborhood of this convex hull.   There is then a $C^ 1 $ nearest point projection $$ \pi_r : \Omega \longrightarrow \dome_r (\Omega),$$ defined by taking a point $z \in\Omega $  to the first point of $\dome_r(\Omega) $ touched by an expanding family of horoballs tangent to $z $ \cite {Epsteinconvex}.  Then:

\begin{theorem} \label {domecomparison} If $\Omega $ is a hyperbolic domain in $\hat \BC $ and $v \in T\Omega $, then
$$  \min\{\sinh (r), 1\} \, | | v | |_T \  \leq\  | |D\pi_r (v) | |_{\BH^ 3 }\  \leq \ e^r  \max\{\sinh (r), 1 \} \,  | | v | |_T. $$
\end{theorem}

This is a variation of the main results in \cite {Bridgemanthurston}, the difference being that Canary and Bridgeman compare $\Omega $ to its dome rather than to $\dome_r (\Omega) $.  However, Theorem \ref {domecomparison} is much easier than their results and its proof avoids all the real work in their paper.  Specifically, their Lemma 4.1 shows that it suffices to prove Theorem \ref {domecomparison} when $\Omega $ is the complement of a finite set of points in $\hat \BC $.  In that case, $\dome (\Omega) $ consists of a finite number of totally geodesic faces that meet at geodesic `ridge lines'.  They show (Lemma 5.1) that $$\pi : \Omega \to \dome (\Omega) $$ is an isometry on the preimage of each face and that the preimage of a ridge line is isometric to the Euclidean product $\BR \times [0,\theta] $, where $\theta $ is the dihedral angle of that ridge.  One must then only notice that a decomposition similar to that of $\Omega $ holds for $\dome_r (\Omega) $ with its path metric: the preimage in $\dome_r (\Omega) $ of a face of $\dome (\Omega) $ under the nearest point projection is part of the surface of points at distance $r$ from that face, and the preimage of a ridge line is a sector of the cylinder consisting of points at distance $r $ from that ridge.  In the first case, the intrinsic metric is an $e^ r $-scale of that on the corresponding face of $\dome (\Omega) $; the intrinsic metric on a cylinder sector of the second case is the Euclidean product $\BR \times [0,\theta \sinh (r)] $.  Theorem \ref {domecomparison} follows easily.

To finish, note that the metrics and projections of the previous page are all preserved by any group of Mobius transformations acting on $\Omega $.  We can then combine Theorems \ref {metriccomparison} and \ref {domecomparison} in the equivariant setting:

\begin{corollary}[Poincar\'e metric vs. $\partial_r CC$] \label {projections}
Let $ N$ be a complete hyperbolic $3$-manifold and assume that every meridian curve in $\partial_c  N $ has length at least $\mu > 0 $ in the Poincar\'e metric. Fix some constant $r > 0 $ and let $$\pi_r : \partial_c  N\to \partial_r \CC ( N) $$ be the nearest point projection onto the boundary of a radius-$r $ neighborhood of the convex core of $ N $.  Then for each tangent vector $v \in T (\partial_c  N) $,
$$ k (r)\, | | v | |_\rho \ \leq \  | |D\pi_r (v) | |_ N \ \leq \ K (r,\mu)\, | | v | |_\rho, $$
where $0 \leq k (r) \leq 1$ and $1 \leq K (r,\mu) \leq \infty$.
\end{corollary}

\subsection{Algebraic convergence in $\Hom (\Gamma,\PSL_2\BC) $}  A hyperbolic structure on an orientable $3$-manifold $M $ corresponds through the holonomy map to a conjugacy class of faithful representations $\pi_1 M\to \PSL_2 \BC $ with discrete and torsion free image. In this section, we briefly review the topology of $(\PSL_2 \BC )$-representation spaces.  Good references for this section are \cite {Benedettilectures} and \cite {Matsuzakihyperbolic}.

Fix a finitely generated group $\Gamma $ and consider the representation variety $\Hom (\Gamma,\PSL_2\BC) $ with its \it algebraic topology: \rm this is the usual term for the compact-open topology, the topology of pointwise convergence.  The following characterization of pre-compact sequences in $\Hom (\Gamma,\PSL_2\BC) $ is well-known and inherent in the work of Culler-Shalen \cite {Cullervarieties}, Morgan-Shalen \cite {Morganvaluations} and Otal \cite {Otalhyperbolization} on compactifications of character varieties.   We give a short proof here for completeness and because the result is elementary at heart.

\begin{lemma}[see also Proposition 4.13, \cite {Kapovichhyperbolic}]\label{boundedtraces}
Let  $\rho_i : \Gamma \to \PSL_2\BC$ be a sequence of representations with pointwise bounded traces: that is, for each $\gamma \in \Gamma $ we have $\sup_i | \tr \rho_i (\gamma) | < \infty $.  Then $(\rho_i) $ can be conjugated to be pre-compact in $\Hom (\Gamma,\PSL_2\BC) $.
\end{lemma}

The absolute value of the trace of an element $A\in \PSL_2\BC $ is twice the hyperbolic cosine of the translation length $\inf_{x \in \BH^ 3} \dist(x, A x) $, so in the above it is equivalent to assume that translation lengths are pointwise bounded.
\begin {proof}
If $X $ is a finite generating set for $\Gamma $, then it suffices to show that there are points $p_i\in \BH^ 3 $ such that
$$\soup_i S^ X_i (p_i) < \infty, \text { where } S^ X_i (p_i) = \sum_{\gamma \in X} \dist \big ( p_i, \rho_i (\gamma) (p_i) \big ).$$

We will proceed by induction on $k $, so assume that $\Gamma = \langle a, b \rangle$ is $2 $-generated.  Since $(\rho_i) $ has pointwise bounded traces, we can choose some $K>0 $ larger than the translation length of any $\rho_i (a) $ or $\rho_i (b) $.  Then the following subsets of $\BH^ 3 $ are always nonempty, and it is easy to prove that they are convex.
$$A_i = \{ x \in \BH^ 3 \ | \ \dist \big (x, \rho_i (a) (x) \big ) \leq K \} $$ $$ B_i = \{ x \in \BH^ 3 \ | \ \dist \big (x, \rho_i (b) (x) \big ) \leq K \}. $$

We are done if we can show that $\soup_i \dist (A_i, B_i) < \infty $.  By assumption, there are points $x_i \in \BH^ 3$ such that $\soup_i \dist (x_i,\rho_i (a b)(x_i)) \leq L < \infty $.  This means that $\dist (\rho_i (b) (x_i), \rho_i (a^ {- 1}) (x_i)) \leq L $.  Now, we also have
$$ \dist (\rho_i (a^ {- 1})(x_i), A_i)=\dist (x_i, A_i) \text { and }\dist (\rho_i (b) (x_i), B_i) =  \dist (x_i, B_i). $$  It follows from hyperbolic geometry (see Figure \ref {boundedtracesfig}) that there is a universal upper bound for the distance from $A_i $ to the midpoint of $[x_i,\rho_i (a^ {- 1}) (x_i)].  $  This has distance at most $L $ from the midpoint of $[x_i,\rho_i (b) (x_i)] $, which has universally bounded distance to $B_i $.  Therefore, $\soup_i \dist (A_i, B_i) < \infty $.
\begin{figure}[tb]
  \begin{center}
  \includegraphics[width=3in]{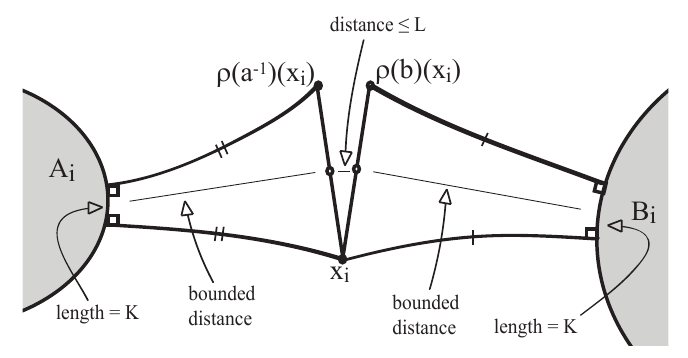}
\caption {$\dist (A_i, B_i)$ is bounded}
  \label{boundedtracesfig}
  \end{center}
\end{figure}
Returning to the general case, assume that the claim is true for $(k- 1 )$-generated groups and that $\Gamma = \langle X \rangle $ for some $k $-element set $X $.  Pick three generators $a, b, c \in X $ and three sequences $(p^ a_i), (p^ b_i),(p^c_i) \in \BH^ 3 $ such that
$$\soup_i S_i^ {X\setminus a} (p^ a_i), \ \soup_i S_i^ {X \setminus b} (p^ b_i), \ \soup_i S_i^ {X \setminus c} (p^ c_i) \ \ < \ \ \infty . $$ Then for every choice of points $(p_i^ {a, b}) $ on the geodesics $[p_i^a,p_i^ b] $, $$\soup_i S_i ^ {X \setminus \{a, b\}}(p_i^ {a, b}) < \infty ,$$ and similarly for sequences of points on the other two edges of the triangle spanned by $\{p_i^ a, p_i^ b, p_i^ c\} $.  However, elementary hyperbolic geometry tells us that there are points $p_i\in\BH^ 3 $ at uniformly bounded distance from all three sides of these triangles.  But then since $X = (X \setminus \{a, b\}) \cup (X \setminus \{b, c\}) \cup (X \setminus \{a, c\}), $ it follows that $\soup_i S_i ^ X (p_i) < \infty $.\end{proof}

Work of J{\o}rgensen and Margulis implies that the subspace $\mathcal D (\Gamma) $ of representations with discrete, torsion free and non-elementary image is closed in $\Hom (\Gamma,\PSL_2 \BC) $.  This is often referred to as Chuckrow's Theorem.  We state here a stronger version that includes a lower semi-continuity law for kernels, and include a proof because it is short and most statements of this result in the literature only apply to faithful representations.

\begin{lemma}[Chuckrow's Theorem]\label {chuckrow} Let $\Gamma $ be a finitely generated group, and let $\tau_i : \Gamma \to \PSL_2\BC $ be a sequence of discrete, torsion-free and non-elementary representations that converges algebraically to a representation $\tau $ with non-abelian image. Then $\tau $ is discrete, torsion-free and for every $\gamma \in \Gamma $, $$\gamma \in \ker \tau \Longrightarrow \gamma \in \ker\tau_i $$ for all $i $ larger than some $i_0= i_0 (\gamma) $.  In particular, when $\ker \tau $ is finitely normally generated within $\Gamma $ we have $\ker \tau \subset \ker \tau_i $ for large $i $.
\label{kernelslemma}
\end{lemma}
\begin{proof}
If $\tau $ is indiscrete, has torsion or violates the condition on kernels, there are sequences  $\gamma_i \in \Gamma $ and $ n_i \to \infty $ such that
\begin {enumerate}
\item for sufficiently large $i $, we have $\tau_{n_i} (\gamma_i) \neq \Id $
\item $\lim_{ i \to \infty } \tau_{n_i} (\gamma_i) = \Id.$
\end {enumerate} 
Note that if there is some $\gamma \in \Gamma$ such that $\tau (\gamma) $ is $k $-torsion, we can choose $(\gamma_i) $ to be the constant sequence $\gamma^ k $ and $n_i = i $. 

 Pick two elements $\alpha_1,\alpha_2 \in \Gamma $ such that $\tau (\alpha_1) $ and $\tau (\alpha_2) $ are isometries of $\Hyp^3 $ of hyperbolic type and have distinct axis.  By $(2) $, we have for sufficiently large $i $ that each of the two pairs $\{\tau_{n_i }(\gamma),\tau_{n_i}(\alpha_k)\} $ violates the Jorgensen inequality (Theorem 2.17 in \cite{Matsuzakihyperbolic}). Therefore, both groups $\left<\tau_{n_i }(\gamma),\tau_{n_i }(\alpha_k) \right>$, $k =1,2 $, are abelian.  But for large $i $, the isometries $\tau_{n_i } (\alpha_1) $ and $\tau_{n_i }(\alpha_2) $ have different axes.  So the only way that both of these groups can be abelian is if $\tau_{n_i }(\gamma) $ is elliptic or trivial. This is a contradiction, since it is nontrivial by assumption and cannot be elliptic since $\tau_{n_i }$ is discrete and torsion-free.
\end{proof}
\section{ Examples }
\label{examples}

We construct here two examples that show that Theorem~\ref{Main} is sharp.  Both examples are pseudo-Anosov maps on the boundary of a genus 3 handlebody. The first extends partially, but not to a handlebody automorphism.  The second does not extend even partially, but its square extends to a handlebody automorphism.  Note that the difficulty here is producing pseudo-Anosov maps; without this constraint producing such examples is an easy exercise.

\begin{example}[Extending only partially]
Let $H $ be the handlebody in Figure \ref {handlebodyfigure} and let $C \subset H $ be the compression body obtained by removing a regular neighborhood of some curve in the interior of $H $ that is isotopic to $c $.  We will construct a pseudo-Anosov map on $\partial H $ that extends to $C $ but not to $H $.

The construction comes from combining the following two results.

\begin {claim}
There is a meridian $\gamma $ for $C $ such that $c $ and $\gamma $ fill $\partial H $.
\end{claim}

\begin{lemma}[Thurston, III.3 in Expos\'e 13 \cite{FLP}] Let $c $ and $\gamma $ be two simple closed curves that fill a surface and let $T_c $ and $T_\gamma $ be the Dehn twists around $c $ and $\gamma $. Then the composition of $T_c \circ T_\gamma^ { -1 } $ is a pseudo-Anosov map.
\label{fillingtwists}
\end{lemma} 
\begin{figure}[tb]
  \begin{center}
  \includegraphics[width=3in]{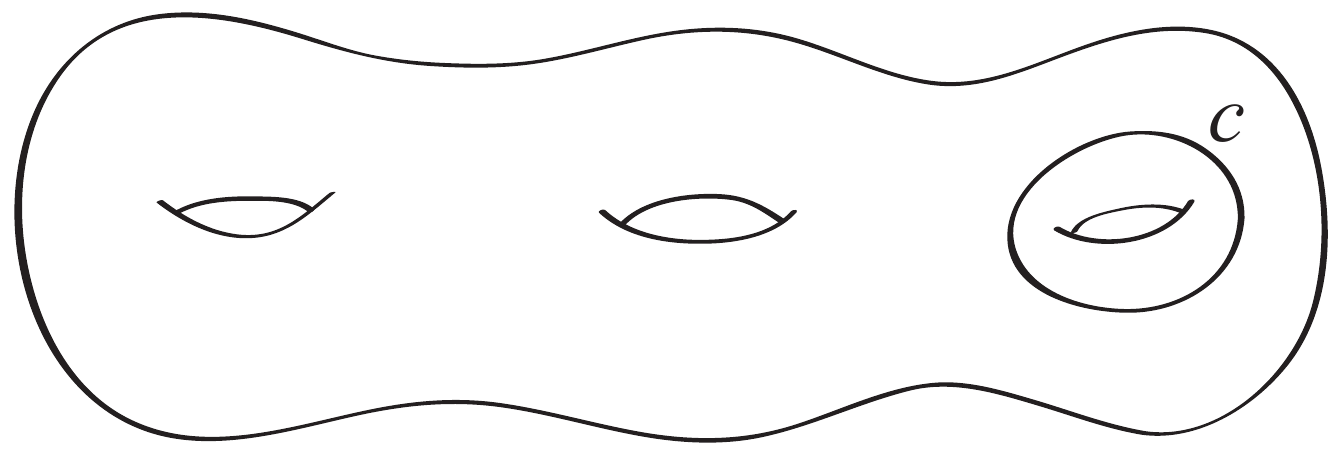}
  \caption{A handlebody $H $ and a curve $c \subset \partial H $.}
  \label{handlebodyfigure}
  \end{center}
\end{figure}

Assuming the claim, it is easy to see that the pseudo-Anosov map  $T_c \circ T_\gamma^ { - 1 } $ extends to $C $ but not to $H $.  First, both $T_c $ and $T_\gamma $ extend to $C $, so the composition does as well.  Second, note that twist $T_\gamma $ extends to $H $.   However, $T_c $ does not extend to $H $ since one can easily construct a meridian $m $ for $H $ with  $T_c (m) $ a non-meridian. So,  $T_c \circ T_\gamma^ { - 1 } $  does not extend to $H $.

It remains to prove the claim.  While there are many ways to do this we use the following lemma, which is readily seen to apply in our situation.

\begin{lemma} \label {fillingmeridians}
If $C $ is a compression body with exterior boundary $\partial_+ C $ and there is a pair of meridians $\alpha,\beta $ on $\partial_+ C $ with $i (\alpha,\beta) > 0 $ and $\alpha $ non-separating, then there is a pseudo-Anosov map of $\partial_+ C $ that extends to $C $.
\end{lemma}

The lemma implies the claim, because iterating this pseudo-Anosov map on any meridian of $C $ gives a sequence of meridians that goes to infinity in the curve complex $C (\partial H) $.  In particular, there is a meridian that fills with $c $.

\begin {proof}
It suffices to show that there is a pair of meridians on $\partial_+ C $ that fill.  For twisting about one and then inverse twisting about the other gives a map that extends  to $C $ and is pseudo-Anosov by Lemma \ref {fillingtwists} above.

Let $f : \partial_+ C \to \partial_+ C $ be a homeomorphism with $f (\alpha) = \alpha $ that is pseudo-Anosov on the complement of $\alpha $.  Such maps exist by the infinite diameter of the complement's curve complex \cite {Masurgeometry1} and Lemma \ref{fillingmeridians}.  Pick a simple closed curve $b $ on $\partial_+ C $ with $i (\alpha,b) = 1 $, and let $b' $ and $\beta' $ be boundary components of regular neighborhoods of $\alpha \cup b $ and $\alpha \cup \beta $, respectively.  Then for large $i $ the distance in the curve complex $C (\partial_+ C \setminus \alpha) $ between $\beta' $ and $f^ i (b') $ is at least $5 $, by \cite [Proposition 7.6]{Masurgeometry2}.  

Note that $f^ i (b') $ is a meridian for $C $, since it is the boundary of a regular neighborhood of $\alpha \cup f^ i (b) $.  We claim that $\beta $ and $f^ i (b') $ fill $\partial_+ C $.  For if a curve $d$ on $\partial_+ C $ intersects neither of these, any boundary component of a regular neighborhood of $\alpha \cup d$ is disjoint from $f^ i (b') $ and a distance at most $2 $ from $\beta' $ in $C (\partial_+ C \setminus \alpha) $.  This can only happen if the distance between $\beta' $ and $f^ i (b') $  in $C (\partial_+ C \setminus \alpha) $ is at most four.
\end{proof}

\end{example}
\begin{example}[Squaring into $\Mod (H) $]
Let $(\Sigma, H^-, H^+)$ be the genus $3 $ Heegaard splitting for $T^ 3$, shown on the left in Figure~\ref{3torusfig}.  The second part of the figure shows an embedded torus that intersects $\Sigma $ in four loops, which appear vertically in the picture.  These loops cobound a pair of annuli in  both $H^ - $ and $H^ + $, so the composition of the Dehn twists along the four loops (in alternating directions) gives a map $f_1 \in\Mod (\Sigma) $ that extends to both handlebodies.  

\begin{figure}[htb]
  \begin{center}
  \includegraphics[width=3.5in]{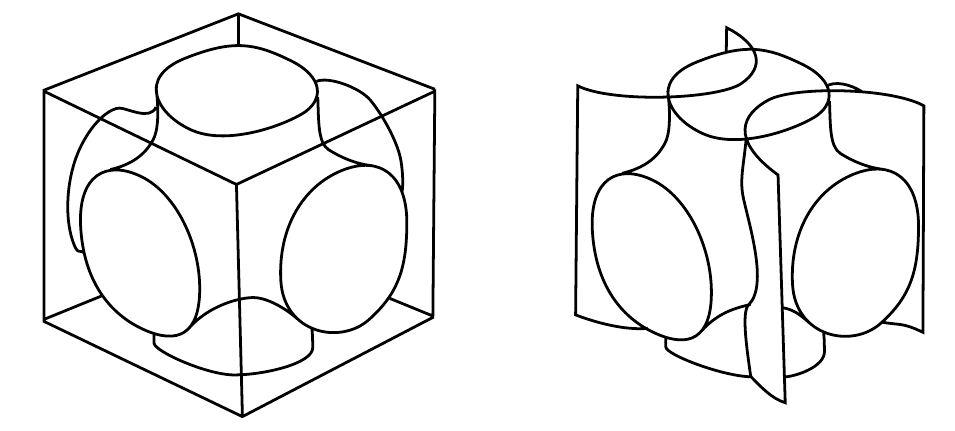}
  \caption{The left shows a Heegaard surface for the 3-torus, restricted to a cube fundamental domain.  The right shows four annuli sharing a common set of four boundary loops in the Heegaard surface.}
  \label{3torusfig}
  \end{center}
\end{figure}

There are  similar collections of loops and automorphisms $f_2, f_3 \in \Mod (\Sigma) $ for the other two coordinate axes in the 3-torus.  The union of all three collections of loops cuts $\Sigma$ into disks, so the subgroup of $\Mod (\Sigma) $ generated by $\{ f_1, f_2, f_3\} $  contains a pseudo-Anosov map $\Phi \in \Mod (\Sigma)$.  This  is essentially an application of Lemma \ref {fillingtwists}, but a clear proof is given in \cite {Ivanovsubgroups}.  Then $\Phi$ extends to both handlebodies, and therefore to an automorphism of $T^3$.

There is an order two automorphism $\Psi \in \Mod (\Sigma)$ that extends to an automorphism of $T^3$ interchanging $H^ - $ and $ H^ + $.   Extending the cube in Figure 2 to a $\BZ^ 3 $-invariant tiling of $\BR^ 3 $, the map $\Psi $ is the projection of a translation by $(\frac 12, \frac 12, \frac 12) $.  In fact, $\Psi $ commutes with $f_1, f_2 $ and $f_3 $ and thus with $\Phi $.  (To see that $\Psi$ commutes with $f_1$, say, note that  $f_1 $ is the composition of Dehn twists in alternating directions along the four loops that appear vertically in the right side of Figure \ref {3torusfig}.   The map $\Psi$ sends each of these four loops to the other of the four that is not adjacent to it, so $\Psi$ preserves the twisting curves of $f_1 $ and the directions of twisting.)

The composition $\Psi \circ \Phi \in \Mod (\Sigma)$ is pseudo-Anosov.  This is because no power $(\Psi  \circ \Phi)^ i $ can fix a simple closed curve $\gamma$ in $\Sigma $, for as $n \to \infty $, $$(\Psi \circ \Phi)^ {2ni }(\gamma) = \Phi ^{2ni }(\gamma) \  \longrightarrow \ \lambda_+ (\Phi ) \in \PML (\Sigma) . $$  The square $(\Psi \circ \Phi)^ 2$ clearly extends to both $H^ + $ and $H^ - $.  However, $\Psi \circ \Phi$ itself does not extend partially to either handlebody.  If it did, since it also interchanges the two there would be a loop on $\Sigma $ that is a meridian in both $H^ - $ and $H^ + $. This contradicts the fact that the genus 3 Heegaard splitting of $T^3$ is irreducible.  To see this, one need only note that $T^ 3 $ is irreducible and has Heegaard genus 3 (the latter follows from the fact that  $\BZ^ 3 $ is not $2 $-generated).  
\end{example}

\section{Marked compression bodies}
\label{compressionbodies}
Recall that if $\Sigma $ is a closed, orientable surface (not $\BS^ 2$), then a \it compression body \rm  is constructed by attaching $2 $-handles to $\Sigma\times [0, 1] $ along a collection of disjoint annuli in $\Sigma \times \{ 0\} $ and $3 $-balls to any boundary components of the result that are homeomorphic to $\BS^ 2 $. It is \it trivial \rm if it is homeomorphic to $\Sigma \times [0, 1] $. The \it exterior boundary \rm $\partial_+ C $ of a nontrivial compression body $C $ is the unique boundary component that $\pi_1 $-surjects; this is $\Sigma \times \{1\} $  in the construction above. The other components of $\partial C $ make up the \it interior boundary \rm $\partial_- C $ and are incompressible in $C $.  We will also sometimes refer to the exterior and interior boundaries of a trivial compression body, with the understanding that they can be chosen arbitrarily.

\begin{lemma}\label {carrying}
Let $M $ be a compact irreducible $3$-manifold with a boundary component $\Sigma $ such that the inclusion $\Sigma \into M $ is $\pi_1 $-surjective.  Then $M $ is a compression body with exterior boundary $\Sigma $. \label {surjectivecompression}
\end{lemma}
\begin {proof}
Bonahon  \cite [Section 2]{Bonahoncobordism} constructed a compression body $C \subset M $ with exterior boundary $\Sigma $ by adjoining to $\Sigma $ a maximal collection of disjoint, properly embedded discs whose boundaries are essential, nonparallel loops in $\Sigma $, and then taking a regular neighborhood and filling in any $\BS^ 2 $ boundary components.  The interior boundary $\partial_- C $ is incompressible in $M $; since $C \into M $ is $\pi_1 $-surjective, Waldhausen's Cobordism Theorem \cite {Waldhausenirreducible} then implies that the components of $M \setminus C $ are all trivial interval bundles. 
\end{proof}

We will often consider compression bodies $C $ that are \it marked \rm by a homeomorphism $\Sigma \to\partial_+ C $ from some fixed surface $\Sigma $.  Lemma \ref{surjectivecompression} indicates that compression bodies are the only manifolds whose fundamental groups can be marked from a boundary component.  In fact, marked compression bodies are determined by the kernels of their marking maps.

\begin{lemma}
Assume that $C_1, C_2 $ are compression bodies with exterior boundaries marked  by homeomorphisms $h_i: \Sigma \to \partial_+ C_i $, $i = 1, 2 $.  Then if $$\ker \big ( \xymatrix { \pi_1 \Sigma \ar[r]^{ (h_1)_*} & \pi_1 C_1 } \big) \ \subset \ \ker \big ( \xymatrix { \pi_1 \Sigma \ar[r]^{ (h_2)_*} & \pi_1 C_2 } \big),$$ the homeomorphism $h_2 \circ h_1^ { -1 } : \partial_+C_1 \to \partial_+ C_2$ extends to an embedding $C_1 \to C_2$.  If $ \ker  \ (h_1)_* = \ker \ (h_2)_* $, then this embedding is a homeomorphism.
\label{kernelsversuscompressionbodies}
\end{lemma}

A \it subcompression body \rm of a compression body $C $ is a submanifold $C' \subset C $ that is a compression body such that $\partial_+ C_1 = \partial_+ C_2 $.  One can then phrase the lemma as saying that $C_1 $ embeds naturally as a subcompression body of $C_2 $.

\begin{proof}
The proof is an easy argument in $3$-manifold topology, so we will try to be brief.  To simplify notation, just consider two compression bodies $C $ and $C' $ that both have boundaries identified to $\Sigma $.  The condition on kernels is that every loop on $\Sigma $ that is null-homotopic in $C $ is also null-homotopic in $C' $.  

As in the proof of Lemma \ref {carrying}, take a maximal collection of disjoint, properly embedded discs $(\BD^2,\partial \BD^2) \into (C,\Sigma) $ whose boundaries are essential, nonparallel loops in $\Sigma $, and let $X \subset C$ be an open regular neighborhood of their union with $\Sigma$.  Then $C \setminus X $ is a union of closed $3 $-balls and the closure of a regular neighborhood of the interior boundary $\partial_- C $ of $C $. 

By assumption, the boundaries of our chosen discs are loops in $\Sigma $ that are null-homotopic in $C' $.  It follows from the Loop Theorem that there is an embedding $h : X \to C' $ that restricts to the identity on $\Sigma $.  Since $C' $ is irreducible, $h $ can be extended to an embedding on all $3 $-ball components of $C \setminus X $.  The remaining components are all trivial interval bundles, so one can map them to regular half-neighborhoods of the corresponding boundary components of $h (X) $.  This produces an embedding $h : C \to C' $.

Now assume that every loop that is null-homotopic in $C' $ is null-homotopic in $C $.  Then as $\partial_- C $ is incompressible in $C $, the surface $h (\partial_- C) $ must be incompressible in $C' $.  Every component of $C' \setminus h (C) $ must then be a trivial interval bundle, for otherwise $\pi_1 (C') $ will decompose as a free product with amalgamation along the corresponding component of $h (\partial_- C) $, preventing $\pi_1 (\Sigma) $ from surjecting onto $\pi_1 (C') $. Therefore $h $ can be stretched near $\partial_- C $ to give a homeomorphism $C \to C' $. 
\end{proof}

\subsection {Markings of hyperbolic compression bodies } Assume now that $C $ is a compression body whose interior $\mathring C $ has a complete hyperbolic metric.  Any marking $h : \Sigma \to \partial_+ C $ then combines with the holonomy map to give a discrete, torsion-free representation $$\rho : \pi_1 \Sigma \to \PSL_2 \BC$$ up to conjugacy, and a diagram that commutes up to homotopy:
$$\xymatrixcolsep{1pc}\entrymodifiers={+!!<0pt,\fontdimen22\textfont2>}\xymatrix{
\Sigma \ar [rrrr]^ { h}   \ar [d] _ -{i      }       &           &         &        &                  \partial_+ C \ar[d] \\ 
N_\rho = \BH^ 3 / \rho (\pi_1 \Sigma) \ar[rrr]^-{\cong} & & & \mathring C \ar [r]^ -\subset & C} $$
Here, $i : \Sigma \to N_\rho $ is any map in the homotopy class  determined by $$\xymatrix  { \pi_1  \Sigma \ar[rr]^-\rho & & \rho(\pi_1\Sigma) \ar[rr] ^ -\cong _* {\txt<50pt>{ \tiny defined up to conjugacy }} & & \pi_1 N_\rho }. $$
We summarize this situation by saying that $\rho : \pi_1 \Sigma \to \PSL_2 \BC$ uniformizes the interior of $C $ and is in the homotopy class of the marking $h : \Sigma \to \partial_+ C $.  

Recall that a \it Bers slice \rm is a space of convex cocompact hyperbolic metrics on $\Sigma \times \BR $ in which one component of the conformal boundary has a fixed conformal structure, while the other varies through $\CT (\Sigma) $.   Bers showed that his slices are precompact in the space of all complete hyperbolic metrics on $\Sigma \times \BR $, \cite {Matsuzakihyperbolic}.  One can define (the closure of) a `generalized Bers slice' as the space of all compression bodies with hyperbolic interior whose exterior boundaries face a convex cocompact end and have some fixed conformal structure $[X] $.  The following is a rigorous formulation of the compactness of such spaces.

\begin{theorem}[Compactness of GBS]  Let $(C_i) $ be a sequence of compression bodies with exterior boundaries marked by homeomorphisms $h_i : \Sigma \to \partial_+ C_i $.   Assume that the interior of $ C_i $ is hyperbolic and uniformized by a representation $\rho_i : \pi_1 \Sigma \to \PSL_2 \BC $ in the homotopy class of $h_i $.  

Suppose that each $\partial_+ C_i $ faces a convex cocompact end of $N_i = \BH^ 3 / \rho_i (\pi_1 \Sigma) $, and is identified with the associated component of the conformal boundary. If there is some $[X] \in\CT (\Sigma) $ with each $(h_i )^*[\partial_+ C_i]= [X] $, then $(\rho_i) $ can be conjugated to be precompact in the representation variety $\Hom (\pi_1 \Sigma,\PSL_2\BC) $. 

If $\rho: \pi_1 \Sigma \to \PSL_2\BC $ is an accumulation point of $(\rho_i) $, then $\rho $ uniformizes the interior of a compression body $C_\rho $ and there is a homeomorphism $h_\rho : \Sigma \to \partial_+ C_\rho $ in the homotopy class of $\rho $.   Moreover, $\partial_+ C_\rho $ faces a convex cocompact end of $N_\rho = \BH^ 3 / \rho (\pi_1 \Sigma)$, and the marking $h_\rho $ can be chosen so that when $\partial_+ C_\rho $ is identified with the conformal boundary of this end, we have $$(h_\rho)^*[ \partial_+ C_\rho ] = [X] \in \CT (\Sigma) . $$\label{compactness}\end{theorem}\begin {proof}
Fix some Poincar\'e metric on $\Sigma $ associated to $[X] \in \CT (\Sigma) $ and homotope the marking  maps $h_i :\Sigma \to \partial_+ C_i $ to be isometries onto the Poincar\'e metrics of their images.  Recall from Section \ref {conformalboundary} that the nearest point projection from the conformal boundary of a hyperbolic $3$-manifold to the boundary of the radius-$1$ neighborhood of its convex core is $C^1$ and infinitesimally bilipschitz with respect to the Poincar\'e metric.  We can then compose the markings and nearest point projections to give a sequence of $C^1 $-embeddings
$$\sigma_i : \Sigma \longrightarrow N_i$$ that are uniformly infinitesimally bilipschitz: for each tangent vector $v \in T\Sigma $,
$$\frac 1 K | | v | |_{\Sigma } \ \leq \ | | D\sigma_i (v) | |_{N_i } \ \leq \ K | | v | |_{\Sigma }. $$  The distortion constant $K\geq 1 $ comes from the nearest point projection and depends only on the injectivity radius of the conformal boundary in the Poincar\'e metric; since our conformal boundaries here always lie in the same Teichm\"uller class, the constant $K $ is independent of $i $.

Every element $[\gamma] \in \pi_1 \Sigma $ is represented by a (based) closed curve $\gamma $ on $\Sigma $.  The bilipschitz bounds above show that the length of $\sigma_i (\gamma) $ in $N_i $ is at most $K $-times the length of $\gamma $ in $\Sigma $.  Therefore, the translation length $$\inf_{x \in \BH^ 3} \dist \big (x,  \rho_i\big ( [\gamma ] \big ) \big ) \leq K \length_\Sigma (\gamma) \ \text { for all } i.$$  Then $(\rho_i) $ is a sequence of representations with bounded pointwise traces, so after conjugating each representation we can assume $(\rho_i) $ is pre-compact in the representation variety $\Hom (\pi_1 \Sigma,\PSL_2 \BC) $, by Lemma \ref {boundedtraces}.

Let $\rho : \pi_1 \Sigma \to \PSL_2 \BC $ be an accumulation point of $(\rho_i) $; in fact, to eliminate double subscripts  let us just assume that $\rho_i \to \rho $ itself.  Chuckrow's Theorem (Lemma \ref {chuckrow}) implies that $\rho $ is discrete and torsion free, so $N_\rho = \BH^ 3 / \rho (\pi_1 \Sigma) $ is a hyperbolic $3$-manifold.  We claim that $(\sigma_i) $ converges to an embedding $$\sigma_\rho : \Sigma \longrightarrow N_\rho $$ whose image bounds a convex subset of $N_\rho $.  From this it will follow that $N_\rho $ has a convex cocompact end with a neighborhood homeomorphic to $\Sigma \times (0,\infty) $.

The map $\sigma_\rho $ is best constructed in the universal cover.  Lift $(\sigma_i) $ to a sequence of $\rho_i$-equivariant maps $\tilde \sigma_i : \BH^ 2 \to \BH^ 3 $. The images $\tilde \sigma_i (\BH^ 2) $ bound convex sets $A_i \subset \BH^ 3 $: each $A_i $ projects to the submanifold of $N_i  $ obtained by removing the neighborhood of $N_i  $'s exterior end that is bounded by $\sigma_i (\Sigma)$.  Because the maps $(\tilde\sigma_i)$ are locally uniformly bilipschitz, after passing to a subsequence they converge to a $\rho$-equivariant local embedding $\tilde \sigma_\rho : \BH^ 2 \to \BH^ 3 $.  The image $\tilde \sigma_\rho (\BH^ 2) $ is the boundary of a $\rho $-invariant convex set $A_\rho \subset \BH^ 3 $, the Hausdorff limit of $(A_i) $. This implies that $\tilde \sigma_\rho $ is a covering map onto its image. 

Passing to the quotient, $\tilde \sigma_\rho $ covers a map $\sigma_\rho : \Sigma \to N_\rho $ whose image bounds a convex subset of $N_\rho $.  Because $\tilde \sigma_\rho $ is a covering map onto its image, the same is true for $\sigma_\rho $.  If $\sigma_\rho $ is a nontrivial covering, there are points $x, y \in \BH^ 2 $ that are not related by a deck transformation of $\Sigma $ but where for some $\gamma \in \pi_1 (\Sigma), $ $$\rho (\gamma)\big(\tilde \sigma_\rho (x)\big) = \tilde\sigma_\rho (y). $$
This shows that $\dist_{\BH^ 3}\big (\rho_i (\gamma)\big[ \tilde\sigma_i (x) \big],\tilde \sigma_i (y)\big) \to 0 .$  Lemma \ref {shortpaths} implies that $$\dist_{\partial A_i }\big (\rho_i (\gamma)\big[ \tilde\sigma_i (x) \big],\tilde \sigma_i (y)\big) \longrightarrow 0,$$ where $\dist_{\partial A_i} $ is the length of a shortest path on $\partial  A_i=\tilde \sigma_i(\BH^ 2)$ connecting $\tilde\sigma_i (x)$ and $\tilde\sigma_i (y) $.  Because the maps $(\tilde \sigma_i ) $ are uniformly  locally bilipschitz covering maps, we can lift these shortest paths to $\BH^ 2 $ to see that for some $\gamma_i \in \pi_1 (\Sigma)$,
$$\dist_{\BH^ 2} (\gamma_i (x), y)\longrightarrow 0 \text { as } i\longrightarrow\infty .$$  This of course implies that $y=\gamma_i (x) $ for large $i $, which is a contradiction.  Therefore, $\sigma_\rho : \Sigma \to N_\rho $ is an embedding.

The image $\sigma_\rho(\Sigma) $ bounds a convex subset of $N_\rho $.  It follows that on the other side $\sigma_\rho (\Sigma)$ bounds a neighborhood of a convex cocompact end of $N_\rho $ that is homeomorphic to $\Sigma \times (0,\infty) $.  The Tameness Theorem of Agol \cite {Agoltameness} and Calegari-Gabai \cite {Calegarishrinkwrapping} implies that $N_\rho $ is homeomorphic to the interior of a compact $3$-manifold $C_\rho $. The map $\sigma_\rho $ is isotopic to a homeomorphism $$h_\rho : \Sigma \to \partial_+ C_\rho $$ onto some boundary component $\partial_+ C_\rho $ of $C_\rho $.  This component carries the fundamental group of $C_\rho $, so Lemma \ref {carrying} implies that $C_\rho $ is a compression body with exterior boundary $\partial_+ C_\rho $.  

Identify $\partial_+ C_\rho $ with the conformal boundary of the end it faces.  Because $\partial_+  C_\rho \to C_\rho $ is $\pi_1  $-surjective,  there is a unique component $O_\rho $ of the domain of discontinuity $\Omega (\rho (\pi_1 \Sigma)) $ that covers $\partial_+ C_\rho $.  It can be described as the set of points in $\partial_\infty \BH^ 3 $ that are endpoints of geodesic rays emanating out of $A_\rho \subset \BH^ 3 $  orthogonal to $\partial A_\rho$.  Similarly, the components $O_i \subset \Omega (\rho_i (\pi_1 \Sigma)) $ that cover $\partial_+ C_i $ consist of all of the endpoints of rays emanating orthogonally from $A_i $.  

The convex sets $( A_i ) $ converge to $A_\rho$ in the Hausdorff topology and the support planes of $A_i $ converge to those of $A_\rho $, so $O_i \to O$ in the sense of Carth\'eodory.  But since $\rho_i \to \rho $, the quotients also converge: $$(h_i)^*[ \partial_+ C_i ] \longrightarrow (h_\rho)^*[ \partial_+ C_\rho ] \in \CT (\Sigma) .$$ Then as $(h_i)^*[ \partial_+ C_i ] = [X] $ for all $i $, $(h_\rho)^*[ \partial_+ C_\rho ]= [X] $ as well.\end{proof}
\begin{figure}[tb]
  \begin{center}
  \includegraphics[width=3.5in]{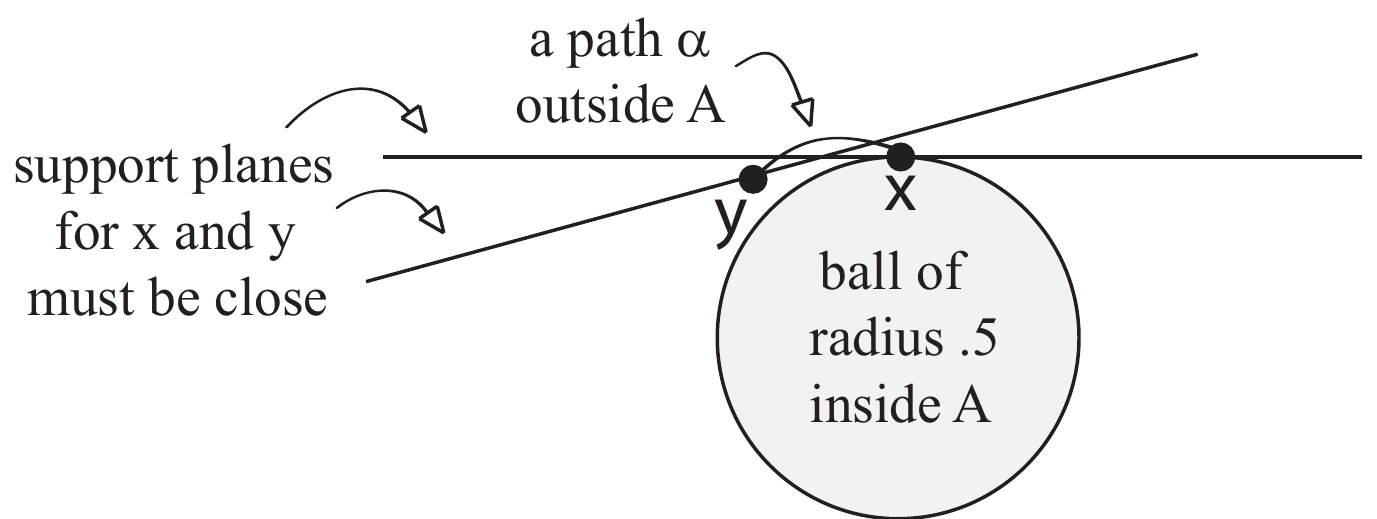}
  \caption{The proof of Lemma \ref {shortpaths}.}
  \label{balls}
  \end{center}
\end{figure}

To finish this section, here is Lemma \ref {shortpaths} promised above.

\begin{lemma}\label {shortpaths}
There is some $\epsilon >0 $ with the following property.  If $A \subset \BH^ 3 $ is the radius-$1 $ neighborhood of a convex set, then for every $x, y \in \partial A $ we have: $$\text {if} \ \dist_{\BH^ 3} (x, y) \leq \epsilon, \ \text {then} \ \dist_{\partial A} (x, y) \leq 2 \dist_{\BH^ 3} (x, y). $$ Here, $\dist_{\partial A} (x, y) $ is the shortest length of a path on $\partial A $ joining $x $ and $y $.
\end{lemma}
\begin {proof}
Every point $x \in \partial A $ llies on the boundary of a  ball of radius  $\frac 12 $ contained in $A $.  Therefore, there is some $\epsilon >0 $ such that if $\dist_{\BH^ 3} (x, y) \leq \epsilon $, any support planes for $x $ and $y $ must intersect at some point at a distance of at most, say, $\dist_{\BH^ 3} (x, y)  $ from both (see Figure \ref {balls}).  We can then make a path $\alpha $ between $x $ and $y $ that does not intersect the interior of $A $ by connecting both $x $ and $y $ to this point of intersection with paths  that run along their support planes.  Projecting the path $\alpha $ onto $\partial A $ does not increase its length, so $\dist_{\partial A} (x, y) \leq 2\dist_{\BH^ 3} (x, y) $.
\end{proof}

\section{A Sequence of Convex-Cocompact Compression Bodies}
\label{sequencesect}

Let $M $ be a compression body with exterior boundary $\partial_+ M = \Sigma $ and let $f : \Sigma \to \Sigma $ be a homeomorphism.  In this section we analyze limits of the following sequence, which is the main tool in the proof of Theorem \ref {Main}.    

\begin{prop}\label {construction}
Given $[X] \in \CT (\Sigma) $, there is a sequence of compression bodies $C_i $ marked by homeomorphisms $h_i : \Sigma \to \partial_+ C_i $, such that for each $i $,
\begin {enumerate}
\item $\ker (\xymatrix  { \pi_1 \Sigma \ar[r]^ {(h_i)_*} & \pi_1 (C_i) }) = f^ { - i }_*(\ker( \pi_1 \Sigma \to \pi_1 M)),$
\item the interior of each $C_i $ has a convex cocompact hyperbolic metric, and when $\partial_+ C_i $ is identified with the conformal boundary of the end it faces, we have $(h_i) ^*[\partial_+ C_i ] = [X]$.
\end{enumerate}
\end{prop}

Everywhere below, $f_* $ will be some fixed automorphism $\pi_1 \Sigma \to \pi_1 \Sigma $ in the homotopy class of $f $.  Note that any two such automorphisms are conjugate,  so act the same way on normal subgroups of $\pi_1 (\Sigma) $.  Therefore, the particular choice does not affect the statement of Proposition \ref {construction}.

\begin {proof}
Fix base points $[X] \in \CT (\Sigma) $ and $[Y] \in \CT (\partial _- M) $.  Using the Ahlfors-Bers Parameterization (Theorem \ref {AhlforsBers}), construct a sequence of convex cocompact hyperbolic metrics $(d_i)$ on the interior of $M $ with conformal boundaries
$$(f^ {i }[X],[Y])\  \in \ \CT (\Sigma) \times \CT (\partial_- M) = \CT (\partial M). $$
It is important to note that $f $ is acting on $\CT (\Sigma) $ by \it pushing forward markings; \rm that is, $f^ i [X] $ is the Teichm\"uller class of the conformal structure on $\Sigma $ obtained by precomposing the charts for $[X] $ with $f^ {- i} $.  (See also Section \ref {Teichmullersection}.)

We define $C_i $ to be the compression body $M $ considered with the metric $d_i $ on its interior,  and mark its exterior boundary with the homeomorphism  $$h_i : \Sigma \longrightarrow \partial_+ C_i $$ obtained by precomposing the equality $\Sigma = \partial_+ C_i $ with $f^ { i } : \Sigma \to \Sigma $.    Then $(h_i) ^*[\partial_+ C_i ] $ is the Teichm\"uller class of the conformal structure whose charts are obtained from those of $f^ i [X]$ by precomposing with $f^ {i} $, so $(h_i) ^*[\partial_+ C_i ] = [X]$.  

Finally, if $\gamma $ is a closed curve on $\Sigma $, we have that 
\begin {eqnarray*} [\gamma] \in \ker (\xymatrix  { \pi_1 \Sigma \ar[r]^ {(h_i)_*} & \pi_1 (C_i)}) &\Longleftrightarrow & h_i (\gamma)\text { is a meridian for } C_i \\ & \Longleftrightarrow & f^ {i} (\gamma) \text { is a meridian for } M \\ & \Longleftrightarrow & [\gamma]\in f^ { - i }_*(\ker( \pi_1 \Sigma \to \pi_1 M)).\end {eqnarray*}  
Therefore, $\ker (\xymatrix  { \pi_1 \Sigma \ar[r]^ {(h_i)_*} & \pi_1 (C_i) }) = f^ { - i }_*(\ker( \pi_1 \Sigma \to \pi_1 M)). $ \end{proof}

As in Section \ref {compressionbodies}, the interior of $C_i $ is uniformized by a representation $$\rho_i : \pi_1 \Sigma \longrightarrow \PSL_2 \BC $$ in the homotopy class of $h_i $.  
Note that from above, the kernel of $\rho_i $ is $$\ker (\rho_i) = f^ { - i }_*(\ker( \pi_1 \Sigma \to \pi_1 M)). $$
By the compactness of generalized Bers slices (Theorem \ref{compactness}), we may assume after conjugation that $(\rho_i) $ is pre-compact in $\Hom (\pi_1 \Sigma,\PSL_2\BC) $.  Then:

\begin{definition} We set $\CA _f \subset \Hom (\pi_1 \Sigma,\PSL_2 \BC)$ to be the subset of the representation variety consisting of all algebraic accumulation points of $(\rho_i) $.  
\end{definition}

Note that the set $\CA_f $ depends on the choice of conjugating sequence used above to make $(\rho_i) $ pre-compact.  However, any other sequence in $\PSL_2\BC $ that conjugates $(\rho_i) $ to be pre-compact in $\Hom (\pi_1\Sigma,\PSL_2\BC) $ differs from our chosen one by a pre-compact sequence in $\PSL_2\BC $.  So, the set of conjugacy classes of representations $\rho \in \CA_f $ does not depend on the conjugating sequence.

Theorem \ref {compactness} gives the following description of points $\rho \in \CA_f $.

\begin{fact}  Every $\rho \in\CA_f $ has discrete and torsion-free image. The quotient $N_\rho = \Hyp^3 / \rho (\pi_1\Sigma)$ is homeomorphic to the interior of a compression body $C_\rho $ whose exterior boundary faces a convex cocompact end of $N_\rho $ and is marked by a homeomorphism $h_\rho : \Sigma \to \partial_+ C_\rho $ in the homotopy class of $\rho $.  
\label {propsequences}
\end{fact}

\label {sequences}
\label{dynamicsect}

In the rest of this section, we show that there is some $\tau \in \CA_f $ such that $C_\tau $ embeds \it naturally \rm in $M $ as a compression body to which a power of $f $ extends. Since each $C_\tau $ comes with a marking $h_\tau : \Sigma \to \partial_+ C_\tau $, a \it natural \rm  embedding $C_\tau \into M $ is just one that restricts to $(h_\tau)^ {-1 } $ on $\partial_+ C_\tau $. 

So, we are looking for some $\tau \in \CA_f $ such that 
\begin {enumerate}
\item $\ker  (\tau) \subset \ker (\pi_1 \Sigma \to \pi_1 M) $, so that $C_\tau $ embeds naturally as a sub-compression body of $M $ (Lemma \ref {kernelsversuscompressionbodies}),
\item some power $f_*^ i : \pi_1 \Sigma \to \pi_1 \Sigma $ preserves $\ker(\tau) $, so that $f  ^ i $ extends to that sub-compression  body of $M $.
\end{enumerate}

We will find $\tau $ by analyzing the dynamics of the action of $f_*$ on the kernels of representations  in $\CA_f $.  First, we must show that there is such an action.

\begin{claim} 
\label{invariantkernels}
The map $f_* : \pi_1 \Sigma \to \pi_1 \Sigma $ acts naturally on the set $$\CK_f =  \{ \ker \rho \ | \  \rho \in \CA_f \}. $$ That is, if $\rho \in \CA_f$ then $f_*(\ker \rho) = \ker \rho'$ for some $\rho' \in \CA_f$. \end{claim}

On the other hand, note that the action of $f _*$ on $\Hom (\pi_1\Sigma,\PSL_2 \BC) $ by precomposition does not usually preserve $\CA_f $, since if $\rho \in \CA_f $ then the (marked) exterior conformal boundary of  $N_\rho $ is $[X]$ while that of $N_{\rho \circ  f} $ is $f^ {- 1} [X] $.  
\begin{proof}
Assume that the subsequence $(\rho_{ i_j }) $ converges to $\rho \in \CA_f $. Passing to a further subsequence, we may assume that $(\rho_{i_j + 1}) $ algebraically converges to some other $\rho' \in \CA_f $. This will be the representation $\rho' $ referenced in the claim.

Observe that there is a homeomorphism $\phi_{ i_j } : C_{ i_j } \to C_{i_j + 1} $ with $\phi_{ i_j } \circ h_{ i_j } = h_{i_j + 1} \circ f $. The restriction $\phi_{ i_j } |_{\partial C_{ i_j } } $ is quasi-conformal with the same dilatation as $f $ has with respect to the conformal structure $[X] $ on $\Sigma $. One can then homotope $\phi_{ i_j } $ on the interior of $C_{ i_j } $ so that it is a $K $-quasi-isometry for some $K $ depending only on $f $ \cite[Theorem 5.31]{Matsuzakihyperbolic}.  We  lift $\phi_{i_j} $ to a $K $-quasi-isometry $$\widetilde {\phi_{ i_j }} : \Hyp^3 \to \Hyp^3 $$ with $\widetilde {\phi_{ i_j }} \circ [\rho_{ i_j } (\gamma)] = \rho_{i_j + 1} \circ f_*(\gamma) $ for all $\gamma \in \Gamma $.  Up to another subsequence, $\widetilde {\phi_{ i_j }} $ converges in the compact open topology to a quasi-isometry $\widetilde{\phi }: \Hyp ^3 \to \Hyp ^3 $, this time satisfying $\widetilde {\phi } \circ \rho (\gamma) = \rho' \circ f_*(\gamma) $. However, from this it is immediate that $\ker \rho' = f_*(\ker \rho) $. The claim follows.
\end{proof}

Recall from $(1) $ that we are searching for elements of $\CK_f $ that are contained in the subset $\ker( \pi_1 \Sigma \to \pi_1 M) \subset \pi_1 \Sigma. $ While not every element of $\CK_f $ has this property, it can be ensured easily with applications of $f_*$. 

\begin {claim}\label {rightkernel}
If $\ker \rho \in \CK_f  $, then there is some $i \in \BZ $ such that  $$f_*^ i (\ker \rho) \subset \ker (\pi_1 \Sigma \to \pi_1 M). $$
\end{claim} 
\begin {proof}
Every $\rho \in \CA_f $ has finitely normally generated kernel: for instance, one can use any subset of $\pi_1 \Sigma $ representing a set of curves that maps to a maximal set of disjoint meridians under $h_\rho : \Sigma \to \partial_+ C_\rho $ (see Fact \ref{propsequences}). It follows from Chuckrow's Theorem (Theorem \ref {chuckrow}) that $$\ \ \ \ker \rho \ \subset  \ker \rho_i = f^{ -i}_*\big(\ker (\pi_1 \Sigma \to \pi_1 M) \,  \big)$$ for some large $i $. Therefore, $f_*^ {i} (\ker \rho) \subset \ker (\pi_1 \Sigma \to \pi_1 M) $.
\end{proof}

To satisfy $(2) $, we must  show that the action of $f_*$ on $\CK_f $ has a finite orbit.  The idea here is to look at \it minimal \rm elements of $\CK_f $, so consider the set
 $$\CK_f^\text {min} = \{ \ker \rho \in \CK_f \ | \ \nexists \ker \rho' \in  \CK_f \text { with }  \ker \rho' \subsetneq \ker  \rho \}.$$

\begin{claim} $\CK_f^\text {min}$ is nonempty, finite and invariant under $f_*$.
\label{minimalkernels}
\end{claim}

\begin{proof}
To show that minimal kernels exist, note that if $\ker \rho \subsetneq \ker \rho' $ then by Lemma \ref{kernelsversuscompressionbodies} the manifold $N_\rho$ must be a strict subcompression body of $N_{\rho'} $. The Euler characteristic of the interior boundary of $N_\rho $ must then be strictly smaller (more negative) than that of $N_{\rho'} $. These Euler characteristics can be no smaller than $\chi(\Sigma)$, so we are guaranteed a compression body whose interior boundary has minimal Euler characteristic, and therefore a minimal kernel.

The $f_* $-invariance follows directly from the definition and Claim \ref{invariantkernels}, so all that remains is to show that $\CK_{ \text {min}  } $ is finite. Assume, hoping for a contradiction, that there is an infinite sequence $\tau_i \in \CA_f $ with pairwise distinct, minimal kernels $\ker \tau_i$. We may assume after passing to a subsequence that $\tau_i $ converges algebraically to some representation $\tau \in \CA_f $. 

So by Chuckrow's Theorem (Lemma \ref{kernelslemma}), $\ker \tau \subset \ker \tau_i $ for large $i$. Minimality of $\ker \tau_i $  implies that this is actually an equality, so for large $i$ all our representations have the same kernel. This is a contradiction.
\end{proof}

\comment {
The claim implies that $f _*$ acts as a permutation on the finite set $\CK_f^\text {min} $.  Any permutation of a finite set has finite order, so we obtain the following corollary.

\begin{corollary}
If $\ker \rho \in \CK_f^\text {min} $, then $f^n_*(\ker \rho) = \ker \rho $ for some $n \in \BN $.  
\label{fixedpoint}
\end{corollary}

If we start with some $\rho \in \CA_f $ with minimal kernel, this argument gives:

Lemma \ref {kernelsversuscompressionbodies} transforms the above corollaries into topology.   Namely, recall that $N_{\tau }= \Hyp^ 3 / \tau (\pi_1 \Sigma) $ is the interior of a compression body $C_{\tau} $.  Corollary \ref {rightkernel} states that $C_{\tau  }$ embeds as a sub-compression body of $M $, and Corollary \ref {fixedpoint} implies that some power of $f: \Sigma \to \Sigma $ extends to $C_{\tau } $.  
In fact,
}

We can now prove the main result of the section.

\begin{theorem}
\label {extending}
Let $M $ be a compression body with exterior boundary $\Sigma $ and let $f : \Sigma \to \Sigma $ be a homeomorphism. Then there is some $\tau\in \CA_f $ such that the associated compression body $C_{\tau} $ embeds naturally in $M $ as a sub-compressionbody to which a power of $f $ extends.  Moreover, up to isotopy $C_\tau $ is the unique maximal sub-compression body of $M $ to which a power of $f $ extends.
\end{theorem}  

Note that the representation $\tau $ may very well be faithful.  In that case, $C_\tau$ is just homeomorphic to $\Sigma \times [0, 1] $ and the assertion that $f $ extends is automatic.  However, in the next section we show that if $f : \Sigma \to \Sigma $ is a pseudo-Anosov homeomorphism with stable lamination in the limit set $\Lambda (M) $ then the compression body $C_\tau $ is actually nontrivial.  


\begin {proof}
To find $\tau $, we first take some $\rho \in \CA_f $ with minimal kernel, as given by Claim \ref {minimalkernels}.  The orbit of $\ker \rho $ under $f_*$ is then finite, and contains the kernel of some $\tau \in \CA_f $ with $\ker  (\tau) \subset \ker (\pi_1 \Sigma \to \pi_1 M) $, by Claim \ref{rightkernel}.  Therefore,
\begin {enumerate}
\item $\ker  (\tau) \subset \ker (\pi_1 \Sigma \to \pi_1 M) $, and
\item some power $f_*^ i : \pi_1 \Sigma \to \pi_1 \Sigma $ preserves $\ker(\tau) $.\end{enumerate}
By Lemma \ref {kernelsversuscompressionbodies}, this implies that $C_\tau $ embeds naturally in $M $ as a compression body to which a power of $f$ extends.

We now prove that $C_\tau $ is maximal among sub-compression bodies of $M $ to which a power of $f $ extends.  The first step is to show that there \it exists \rm a  sub-compression body of $M $ to which a power of $f $ extends that is maximal up to isotopy. As there can be no strictly increasing infinite sequence of sub-compressionbodies of $M $, it suffices to prove the following claim.

\begin {claim}
If $C_1 $ and $C_2 $ are sub-compression bodies of $M $ to which powers of $f $ extend, then there is some sub-compression body $C \subset M $ to which a power of $f $ extends that contains isotopes of both $C_1 $ and $C_2 $.
\end{claim}
\begin{proof}
 Suppose that $f^j$ extends to $C_1 $ and $f ^ k $ extends to $C_2 $.  Then $f_\star^ {kj}$ preserves the kernels of $\pi_1 \Sigma \to C_1 $ and $\pi_1 \Sigma \to C_2 $, which are then subsets of $\ker \rho_{i kj} $ for all $ i $.   Applying the first part of Theorem \ref {extending} to $f^ {kj} $, there is some accumulation point $\rho $ of $(\rho_{i kj}) $ where $C_\rho $ embeds naturally in $M $ as a sub-compression body to which some $f^ {i kj} $ extends.  But $$\ker (\pi_1 \Sigma \longrightarrow \pi_1C_1), \ \kernel (\pi_1 \Sigma \longrightarrow \pi_1 C_2) \ \subset \  \ker \rho, $$ so by Lemma \ref {kernelsversuscompressionbodies} the image of $C_\rho $ in $M $ contains isotopes of $C_1 $ and $C_2 $.  
\end {proof}

We now show that the compression body $C_\tau $ is at least as compressed, in the sense of having an interior boundary with less negative Euler characteristic, as any other sub-compression body of $M $ to which a power of $f $ extends.  This will show that $C_\tau $ is in fact the maximal  sub-compression body referenced above.

To prove this, we rely upon the following claim.

\begin {claim}\label{thatclaim}
Let $C $ be a sub-compression body of $M $ to which some power $f^ k $ of $f $ extends.  Then if $\rho \in  \CA_f $, there is some $l \in \BZ $ such that $$ f_*^ {-l}\big (\kernel (\pi_1 \Sigma \to \pi_1 C) \, \big )\subset \ker  (\rho).$$
\end{claim}
\begin {proof}
Since $\rho \in \CA_f $ there is a subsequence $(\rho_{i_n} )$ of $(\rho_i) $ that converges to  it; passing to another subsequence if necessary, we may assume that the indices $i_n $ all lie in some fixed mod-$k$ equivalence class $[ l ] \subset \BN $.  Then for each $i_n $,
\begin {align*}  f_*^ {-l}\big (\kernel (\pi_1 \Sigma\to \pi_1 C) \,\big )& =  f_*^ {- i_n}\big ( \kernel (\pi_1 \Sigma\to \pi_1 C)\, \big) \\
& \subset  f_*^ {- i_n}\big ( \kernel (\pi_1 \Sigma\to \pi_1 M)\, \big) \\
& = \kernel (\rho_{i_n}),
\end {align*}
by Proposition \ref {construction}.  Taking the limit as $n\to \infty $ proves the claim.
\end{proof}
Lemma \ref {kernelsversuscompressionbodies} and Claim \ref{thatclaim} imply that any $C\subset M $ to which a power of $f $ extends can be embedded as a sub-compression body of $C_\tau $.  Therefore $C_\tau $ is at least as compressed as $C $, which implies as above that $C_\tau $ is the (unique) maximal sub-compression body of $M $ to which a power of $f $ extends.
\end{proof}

\section{Stable Laminations and the Proof of Theorem \ref {Main}}
\label{laminationsect}

In this section we will analyze the set of accumulation points $\CA_f $ introduced in Section \ref{sequences} in the case that $f $ is a pseudo-Anosov map.  The main result is the following; after proving it we will quickly derive Theorem \ref{Main}.

\begin{prop}
Let $M $ be a compression body with exterior boundary $\Sigma $. If $f: \Sigma \to \Sigma $ is a pseudo-Anosov map whose stable lamination $\lambda_+ (f) $ lies in the limit set $\Lambda (M),$ then every $\rho \in \CA_f$ has a non-trivial kernel. 
\label{nontrivialkernels}
\end{prop}

We will actually prove the contrapositive: that if some $\rho \in \CA_f $ is faithful then $\lambda_+ (f) \notin \Lambda (M) $.  The first step in the argument is to show that faithful representations in $\CA_f $ can have no parabolics.  Since it involves no extra effort, we prove the following stronger statement.

\begin{lemma} 
If $\rho \in \CA_f$ is faithful, then $N_\rho = \Hyp^ 3 / \rho (\pi_1 \Sigma) $ is homeomorphic to $ \Sigma \times \BR$ and has no cusps.  One of the ends of $N_\rho $ is convex cocompact and the other is degenerate with ending lamination $\lambda_- (f) $.
\end{lemma}

\begin{proof}
Recall from Proposition \ref{propsequences} that $N_\rho $ is homeomorphic to a compression body whose exterior boundary is homeomorphic to $\Sigma $ through a map in the homotopy class determined by $\rho $.  Since $\rho $ is faithful, $N_\rho $ must be homeomorphic to $\Sigma \times \BR $.  Proposition \ref {propsequences} also states that one end of $N_\rho $ is convex cocompact.  We claim that the other end is degenerate and that its ending lamination is the \it unstable \rm lamination $\lambda_-(f) $.  

Assume that $\rho $ is the limit of some subsequence $(\rho_{ i_j } )$ of the sequence $(\rho_i) $ whose accumulation points comprise $\CA_f $. 
Since $\lambda_-(f) $ is a full lamination with no closed leaves, it suffices by \cite [Prop 9.7.1] {Thurstongeometry} to show that it is unrealizable by a pleated surface in the homotopy class determined by $\rho $.  Fix a meridian curve $\gamma $ on $\Sigma =\partial M$. By \cite [Theorem 5.7]{Bleilerautomorphisms}, after passing to a subsequence we may assume that $f ^ {- i_j } (\gamma) $ converges in the Hausdorff topology to some lamination $\lambda_H \subset \Sigma $ that is the union of $\lambda_- (f) $ and finitely many leaves spiraling onto it. If $\lambda_-(f) $ is realizable in $N_\rho $, then \cite [Theorem 2.3] {Brockcontinuity} implies that $\lambda_H $ is as well. So, in search of a contradiction, assume that $\lambda_H $ is realizable in $N_\rho $ by a pleated surface in the homotopy class determined by $\rho $.

By Lemma 4.5 in \cite {Brockcontinuity} there is a train track $\tau $ in $\Sigma $ that carries $\lambda_H $ and a smooth map $f : \Sigma \to N_\rho $  in the homotopy class of $\rho $ that maps every train path on $\tau $ to an immersed path in $N_\rho $ with geodesic curvatures less than some $\epsilon < 1 $. In the terminology of \cite{Brockcontinuity}, $\tau $ is an $\epsilon $-nearly straight train track in $N_\rho $. Now, for large $i_j $ algebraic convergence gives us  immersions $$\Phi_{ i_j } : U \to N_{ i_j } $$ defined on a neighborhood $U \supset f (\Sigma) $ such that  (see Lemma 14.18 in \cite {Kapovichhyperbolic})
\begin {enumerate}
\item $\Phi_{i_j} $ converges to a local isometry in the $C^ k $-topology, for any $k \in\BN $,
\item the composition $\Phi_{ i_j } \circ f $ is in the homotopy class determined by $\rho_{ i_j } $.
\end{enumerate} Then for large $i_j $, the image $\Phi_{ i_j } \circ f (\tau) $ is an $\epsilon' $-nearly straight train track in $N_{ i_j } $ for some $\epsilon' <1 $. If $i_j $ is suitably large, the curve $f ^ {- i_j } (\gamma) $ is carried by $\tau $, and therefore $h_{ i_j }\circ f ^ {- i_j } (\gamma) $ has a realization in $N_{ i_j } $ with all geodesic curvatures less than $\epsilon' <1 $. This is impossible, because it is null-homotopic in $N_{ i_j } $. 
\end{proof}

The second step in the proof of Proposition \ref {nontrivialkernels} is a pleated surfaces argument.  For any simple closed curve $\beta \subset \Sigma $, let $\beta_\rho $ be the closed geodesic in $N_\rho $ with holonomy $\rho (\beta) $.  If $\rho $ is faithful, one can bound the distance in $N_\rho $ between $\beta_\rho $ and a fixed $\alpha_\rho $ by the distance $\dist_{\C}(\alpha,\beta) $ in the curve complex $\mathcal C (\Sigma) $:

\begin{lemma}
\label{distanceboundslem}
Assume that $\rho \in \CA_f $ is faithful and let $\alpha \subset \Sigma $ be a simple closed curve.  Then for every $k \in \BN $, there is a constant $K =K (\rho,\alpha,k) $ such that for any other simple closed curve $\beta \subset \Sigma $, we have $$\dist_C(\alpha, \beta) \leq k \ \Longrightarrow \ \dist_{ N_\rho }(\alpha_\rho, \beta_\rho) \leq K . $$
\end{lemma}

Here, the distance $\dist_{ N_\rho } $ between two subsets of $N_\rho$ is simply the infimum of the distances between points in one and points in the other.

\begin{proof}
We proceed by induction.  The base case $k=0 $ is trivial, so assuming that there is some $K =K (\rho,\alpha,k) $ for which the claim holds for $k $, we will attempt to find a similar constant for $k+1 $.

Assume $\dist_C (\alpha,\beta) = k+ 1 $ and choose some curve $\gamma $ disjoint from $\beta $ with $\dist_C(\alpha,\gamma) =k. $ Since $\gamma$ and $\beta$ are disjoint and $N_\rho $ has no cusps, \cite[Lemma 6.12]{Matsuzakihyperbolic} implies that there is a pleated surface in the homotopy class of $\rho $ that realizes both $\gamma $ and $\beta $.  By the induction hypothesis, the geodesic realization $\gamma_\rho $ lies at a distance at most $K $ from $\alpha_\rho $.  The space of pleated surfaces in $N_\rho $ that intersect the $K $-ball around $\alpha_\rho $ is compact \cite [Lemma 6.13] {Matsuzakihyperbolic}, so this puts an upper bound $K' = K'(\rho,\alpha,k)$ on the distance between the geodesic realizations $\beta_\rho $ and $\gamma_\rho $ (even better, between $\beta_\rho $ and the part of $\gamma_\rho $ that lies at distance $K $ from $\alpha_\rho $).  Thus if $d_C(\alpha, \beta) = k + 1$ then $d_{ N_\rho }(\bar \alpha, \bar \beta) < K+K'$.  
\end{proof}

We can now finish the proof of Proposition \ref {nontrivialkernels}.

\begin{proof}[Proof of Proposition \ref {nontrivialkernels}]
We will prove the contrapositive.  Assume that $\rho\in \CA_f $ is faithful and that it is the limit of some subsequence $(\rho_{i_j}) $ of $(\rho_i) $.  Fix an essential loop $\alpha $ in $\Sigma $.  It follows from Lemma 14.28 in \cite {Kapovichhyperbolic} that for every $L,\epsilon >0 $, we have for sufficiently large $i_j$ a $(1+\epsilon) $-bilipschitz immersion $$\Phi_{i_j} : \CN_L (\alpha_\rho) \longrightarrow N_{ i_j }, $$ where $\CN_L (\alpha_\rho) $ is the radius $L$-neighborhood of $\alpha_\rho$ in $N_\rho$.  Moreover, the map $\Phi_{ i_j } $ is  compatible with our markings: its composition with a map $\Sigma \to N_\rho $ in the homotopy class of $\rho $ is a map $\Sigma \to N_{\rho_{ i_j } } $ in the homotopy class of $\rho_{ i_j } $.

 Lemma \ref{distanceboundslem} implies that given $k>0$, there is some such $\Phi_{ i_j } $ whose domain contains the geodesic representative $\gamma_\rho $ of any curve $\gamma $ with $\dist_{ C } (\alpha,\gamma) \leq k $.  As long as the bilipschitz constant of $\Phi_{ i_j } $ is very small, the image $\Phi_{i_j }(\gamma_\rho)$ will be a closed curve in $N_{\rho_{ i_j } } $ in the homotopy class of $\rho_{ i_j } (\gamma) $ with geodesic curvatures less than $1$.  This curve is then homotopically essential in $N_{i_j}$.  So, for every $k >0 $ we have for sufficiently large $i_j $ that $$\dist_C (\alpha,\gamma) < k \ \ \Longrightarrow \ \gamma \notin \ker \rho_{ i_j } .$$
Geometrically, this means that in the curve complex $\mathcal C (\Sigma) $ the set of curves that lie in $\ker \rho_{ i_j} $ becomes farther and farther away from $\alpha $ as $i_j \to \infty $. However, we saw in Proposition \ref {propsequences} that $$\ker \rho_{ i_j } = f_\star^ {- i_j } \big(\ker ( \pi_1 \Sigma \to \pi_1 M) \big),$$ so the set of simple closed curves lying in $\ker \rho_{ i_j } $ is exactly the image $f^ { -i_j } ( \mathcal D(M))$ of the disk set of $M $.  Composing the entire picture with $f^ { i_j } $, $$\dist_{\mathcal C (\Sigma) } \big(f ^ {  i_j } (\gamma), \mathcal D (H)\big) \to \infty \ \text { as } \ i_j \to \infty. $$
Applying Proposition \ref {boundeddistance}, the stable lamination $\lambda_+ (f)$  cannot lie in $\Lambda (M) $.
\end{proof}

At this point, Theorem \ref {Main} follows from applying the machinery we have built. Recall the statement given in the introduction.

\begin{named}{Theorem \ref{Main}}
Let $f \in \Mod (\Sigma) $ be a pseudo-Anosov map on some boundary component $\Sigma $ of a compact, orientable and irreducible $3 $-manifold $M $.  Then the (un)stable lamination of $f $ lies in $\Lambda (M)$ if and only if $f $ has a power that partially extends to $M$.
\end{named}

\begin{proof}
The `if' direction is trivial.  If $f^ i $ extends to a nontrivial sub-compression body $C \subset M $, then any meridian $\gamma $ for $C $ gives sequences $(f^{ k i} (\gamma)) $ and $(f^{ -k i} (\gamma)) $ of meridians that converge to the stable and unstable laminations of $f $.

Assume now that the stable lamination of $f $ lies in $\Lambda (M)$; the same argument will work for unstable laminations if we first invert $f $. As mentioned in the introduction, we can assume without loss of generality that $M $ is a compression body with exterior boundary $\Sigma $. Build the sequence of representations $(\rho_i) $ as we did in Section \ref {sequences}.  By Corollary \ref{extending}, some power $f^ k$ extends to a subcompression body $C \subset M $ that is homeomorphic to $N_\rho $ for some algebraic accumulation point $\rho$ of $(\rho_i) $.  Proposition \ref{nontrivialkernels} implies that $\rho $ must have a nontrivial kernel.  So, the compression body $C $ cannot be trivial, implying that $f^ k$ partially extends to $M $.
\end{proof}

\bibliographystyle{amsplain}
\bibliography{manifolds}

\end{document}